\documentclass[english,10pt]{amsart}
\usepackage[english]{babel}

\usepackage[matrix,arrow]{xy}
\xyoption{all}
\usepackage{amscd,amssymb,amsfonts,amsmath}
\usepackage[T1]{fontenc}
\usepackage[utf8]{inputenc}
\usepackage{xcolor}
\usepackage{graphics}
\usepackage{epsfig}
\usepackage{mathrsfs}
\definecolor{violet}{rgb}{0.0,0.2,0.7}
\definecolor{rouge2}{rgb}{0.8,0.0,0.2}
\usepackage{hyperref}
\hypersetup{
    bookmarks=true,         
    unicode=false,          
    pdftoolbar=true,        
    pdfmenubar=true,        
    pdffitwindow=false,     
    pdfstartview={FitH},    
    pdftitle={},    
    pdfauthor={},     
    colorlinks=true,       
   linkcolor=violet,          
    citecolor=rouge2,        
    filecolor=black,      
    urlcolor=cyan}           

\newcommand\mypagesizel{
\textwidth= 6.5in
\textheight=9in
\voffset-.55in
\hoffset -0.75in
\marginparwidth=56pt
}

\newcommand{\Alb}{\textup{Alb}}

\newcommand{\Pic}{\textup{Pic}}

\newcommand{\codim}{\textup{codim}}

\newcommand{\Hol}{\textup{Hol}}
\newcommand{\tr}{\textup{tr}}

\renewcommand{\phi}{\varphi}
\newcommand{\into}{\hookrightarrow}
\newcommand{\map}{\dashrightarrow}

\newcommand{\wt}{\widetilde}
\newcommand{\wh}{\widehat}
\newcommand{\wb}{\overline}

\renewcommand{\le}{\leqslant}
\renewcommand{\ge}{\geqslant}

\newcommand{\bP}{\mathbb{P}}

\newcommand{\sA}{\mathscr{A}}
\newcommand{\sB}{\mathscr{B}}
\newcommand{\sC}{\mathscr{C}}

\newcommand{\sE}{\mathscr{E}}

\newcommand{\sG}{\mathscr{G}}

\newcommand{\sL}{\mathscr{L}}

\newcommand{\sO}{\mathscr{O}}

\newcommand{\sT}{T}

\newcommand{\sX}{\mathscr{X}}
\newcommand{\sY}{\mathscr{Y}}
\newcommand{\sZ}{\mathscr{Z}}

\renewcommand{\P}{\mathbb{P}}
\newcommand{\vp}{\varphi}

\renewcommand{\ge}{\geqslant}
\renewcommand{\le}{\leqslant}

\newcommand{\om}{\omega}

\newcommand{\ddc}{dd^c}

\DeclareMathOperator{\reg}{reg}

\DeclareMathOperator{\Sp}{Sp}
\DeclareMathOperator{\SU}{SU}

\numberwithin{equation}{section}

\newcommand{\X}{\mathscr X}

\newcommand{\Xr}{X_{\reg}}
\renewcommand{\L}{\mathscr L}

\mypagesizel

\newtheorem{thm}{Theorem}[section]
\newtheorem*{thma}{Theorem A}
\newtheorem*{thmb}{Theorem B}

\newtheorem*{thmf}{Corollary C}

\newtheorem{lemma}[thm]{Lemma}
\newtheorem{cor}[thm]{Corollary}

\newtheorem{prop}[thm]{Proposition}

\newtheorem{conj}[thm]{Conjecture}

\newtheorem*{thm*}{Theorem}
\theoremstyle{definition}
\newtheorem{defn}[thm]{Definition}

\newtheorem{exmp}[thm]{Example}

\newtheorem{defn-thm}[thm]{Definition-Theorem} 
\newtheorem{defn-lemma}[thm]{Definition-Lemma}

\newtheorem{rem}[thm]{Remark}
\newtheorem{fact}[thm]{Fact}

\theoremstyle{remark}

\newtheorem*{not-and-def}{Notation and definitions}
\newtheorem{setup-and-notation}[thm]{Setup and Notation}
\newtheorem{additional-assumption}[thm]{Additional assumption}
\newtheorem{notation}[thm]{Notation}

\begin{document} 

\title[A decomposition theorem for smoothable varieties with trivial canonical class]{A decomposition theorem for smoothable varieties with trivial canonical class}

\author{St\'ephane \textsc{Druel}}
\address{St\'ephane Druel: Institut Fourier, UMR 5582 du CNRS, Universit\'e Grenoble Alpes, CS 40700, 38058 Grenoble cedex 9, France} 
\email{stephane.druel@univ-grenoble-alpes.fr}

\author{Henri Guenancia}
\address{Henri Guenancia: Department of Mathematics, 
Stony Brook University, Stony Brook, NY 11794-3651, United States}
\email{henri.guenancia@stonybrook.edu}


\begin{abstract}
In this paper we show that any smoothable complex projective variety, smooth in codimension two, with klt singularities and numerically trivial canonical class admits a finite cover, \'etale in codimension one, that decomposes as a product of an abelian variety, and singular analogues of irreducible Calabi-Yau and irreducible symplectic
varieties.
\end{abstract}

\date{\today}
\maketitle

{\small\tableofcontents}

\section{Introduction}

The Beauville-Bogomolov decomposition theorem asserts that any compact K\"ahler manifold with numerically trivial canonical bundle admits an \'etale cover that decomposes into a product of a torus, and irreducible,
simply-connected Calabi-Yau, and symplectic manifolds (see \cite{beauville83}). 

With the development of the minimal model program, 
it became clear that singularities arise as an inevitable part of higher dimensional life.
If X is any complex projective manifold with Kodaira dimension $\kappa(X) = 0$, standard
conjectures of the minimal model program predict the existence of a birational contraction
$X \map X_{\rm min}$, where $X_{\rm min}$ has terminal singularities and $K_{{X}_{\rm min}}\equiv 0$.
This makes it imperative to extend the Beauville-Bogomolov decomposition theorem to the singular setting.

In the singular setting, several notion of irreducible Calabi-Yau varieties or irreducible symplectic varieties have been proposed but the results of \cite{GKP} and \cite{GGK}
provide evidence that the following definition should be the correct one in view of a singular analogue of the Beauville-Bogomolov decomposition theorem.

\begin{defn}
\label{def:cy}
Let $X$ be a normal projective variety with canonical singularities such that $K_X \sim_{\mathbb Z} 0$.  
\begin{enumerate}
\item We call $X$ \emph{irreducible Calabi-Yau} if
$h^0 \bigl(Y, \Omega_{Y}^{[q]} \bigr) = 0$ for all numbers
$0 < q < \dim X$ and all finite covers $Y \to X$, \'etale in codimension one.
\item We call $X$ \emph{irreducible symplectic} if there
exists a holomorphic symplectic $2$-form
$\sigma \in H^0 \bigl(X, \Omega_X^{[2]} \bigr)$ such that for all finite étale
covers $f\colon Y \to X$, \'etale in codimension one, the exterior
algebra of global reflexive forms is generated by $f^{[*]}\sigma$.
\end{enumerate}
\end{defn}
\noindent Let $X$ be a normal projective variety of dimension at least $2$ with $K_X \equiv 0$ and klt singularities.
Suppose moreover that its tangent sheaf is strongly stable in the sense of Definition~\ref{def:strongStab}.
In \cite[Theorem C]{GGK}, it is proved that $X$ admits a finite cover, \'etale in codimension one, that is either an irreducible Calabi-Yau variety or an irreducible symplectic variety. That result, given the infinitesimal version of the Beauville-Bogomolov decomposition theorem proved in \cite{GKP}, is a strong indication that the notions of irreducible Calabi-Yau and symplectic varieties described in Definition~\ref{def:cy} should be the good ones. 

In \cite{bobo} the first author extends the Beauville-Bogomolov decomposition theorem to complex projective varieties of dimension at most five with klt singularities and numerically trivial canonical class.
The main result of our paper is the following decomposition theorem for smoothable mildly singular spaces with numerically trivial canonical class.

\begin{thma}\label{thm:main}
Let $X$ be a normal complex projective variety with klt singularities and smooth in codimension two. Suppose that $K_X \equiv 0$.
Suppose furthermore that there exists a flat projective holomorphic map with connected fibers $f \colon \sX \to \Delta$ from a 
normal analytic space $\sX$
onto the complex open unit disk $\Delta$ such that 
$X \cong f^{-1}(0)$ and such that $f^{-1}(t)$ is smooth for $t \neq 0$.
Then there exists a finite cover $Y \to X$, \'etale in codimension one, and a decomposition
of $Y$ into a product of an abelian variety and irreducible, Calabi-Yau and symplectic varieties.
\end{thma}

\begin{rem}\label{rem:schl}
The assumptions of Theorem A imply that $X$ does not have any quotient singularity. Indeed, a theorem of Schlessinger \cite[Theorem 2]{Sch71} (see also  \cite[Theorem 10.1]{artin_def_sing}) shows that a germ of a quotient singularity $(X,x)$ is rigid as soon as $\codim \overline{ \{x\} }\ge 3$. 
\end{rem}

In addition to the smoothability condition, the strategy of proof of Theorem~A requires us to assume that the $\codim(X\smallsetminus X_{\rm reg})\ge 3$. Let us briefly explain why. The idea of the proof is to consider a cover of the smooth generic fiber that splits off an abelian variety as well as irreducible, simply-connected Calabi-Yau and symplectic manifolds. A significant part of the paper is devoted to showing that one can take the flat limits of these irreducible pieces and recover the central fiber as product of those limits.

It is then tempting to believe that the flat limit $X$ of irreducible and simply-connected, Calabi-Yau or symplectic manifolds 
admits a finite cover, \'etale in codimension one, which is an irreducible Calabi-Yau or symplectic variety. Unfortunately, this turns out to be false in general as we explain in Section~\ref{sec:nonsstable}. This makes it much harder to use the smoothability assumption in order to prove a decomposition theorem in full generality. 
However, we are able to prove that $T_X$ is stable, see Theorem~B below. 
To conclude the proof of Theorem~A, we show that, in the setting of Theorem~A, we must have $\pi_1^{\textup{\'et}}\big(X_{\textup{reg}}\big)=\{1\}$
(see Theorem~\ref{thm:fundamental_group}).
Note that in Theorem~A, we do not require $\sX$ to be $\mathbb Q$-Gorenstein as this condition is automatically satisfied, see Lemma~\ref{lemma:q_gorenstein}.

As we explained above, it seems difficult to obtain a full decomposition theorem using our strategy without further assumptions on the singularities of $X$. However, we are still able to produce a splitting of some quasi-étale cover of $X$ where each non-abelian factor has a stable tangent bundle, and possesses the same algebra of reflexive holomorphic forms as the one of an irreducible, simply-connected, Calabi-Yau or symplectic manifold of the same dimension.

\begin{thmb}
\label{thm:main2}
Let $X$ be a normal complex projective variety with klt singularities. Suppose that $K_X \equiv 0$.
Suppose furthermore that there exists a projective morphism with connected fibers $f \colon \sX \to \Delta$ from a 
normal analytic space $\sX$ whose canonical divisor $K_\sX$ is $\mathbb{Q}$-Cartier
onto the complex open unit disk $\Delta$ such that 
$X \cong f^{-1}(0)$ and such that $f^{-1}(t)$ is smooth for $t \neq 0$.
Then, there exists an abelian variety $A$
as well as a projective variety $X'$ with canonical
singularities, a finite cover $$A \times X' \to X,$$ \'etale in
codimension one, and a decomposition
$$X' \cong \prod_{i\in I} Y_i \times \prod_{j\in J} Z_j$$
of $X'$ into normal projective varieties with trivial canonical class, such that the following holds.

\begin{enumerate}
\item The sheaf $T_{Y_i}$ is stable with respect to any polarization and one has $h^0 \big( Y_i, \Omega_{ Y_i}^{[q]} \big) = 0$ for all numbers $0 < q < \dim Y_i$.
\item The sheaf $T_{Z_j}$ is stable with respect to any polarization and there exists a reflexive $2$-form $\sigma \in
H^0\big(Z_j, \Omega_{Z_j}^{[2]} \big)$ such that $\sigma$ is everywhere
non-degenerate on the smooth locus of $Z_j$, and such that the exterior algebra of global reflexive forms is generated by $\sigma$.\\
\end{enumerate}
\end{thmb}

\begin{rem}\label{rem:nami}
If $\pi:\wt Z_j \to Z_j$ is a $\mathbb Q$-factorial terminalization of the variety $Z_j$ from Theorem B, then it follows from \cite[Corollary 2]{Namikawa06} that $\wt Z_j$ is a \textit{smooth} symplectic variety, but not necessarily irreducible, since the varieties $Z_j$ are smoothable by construction (see Section~\ref{sec4}). In particular, 
$Z_j$ admits a symplectic resolution. Note that the existence of $\pi$ is established in \cite[Corollary 1.4.3]{bchm}.
\end{rem}

In fact, a little more can be said about the factors of $X'$ in the decomposition given by Theorem~B above. We refer to Section~\ref{ssec:factors} for partial results from the point of view of holonomy representation, and for a proof of Conjecture~\ref{conj} below assuming that a weak analogue of Beauville-Bogomolov decomposition theorem holds.

\begin{conj}\label{conj}
Let $X$ be a normal complex projective variety with klt singularities and $K_X \equiv 0$.
Suppose that $T_X$ is stable with respect to some polarization. 
Then there exists a quasi-\'etale cover $Y \to X$ such that either $Y$ is an abelian variety, or
it splits as a product of copies of a single Calabi-Yau (resp. irreducible symplectic) variety.
\end{conj}
Note that a positive answer to the conjecture above implies that Theorem A holds without the assumption that $X$ is smooth in codimension $2$. 

The following result is an immediate consequence of Theorem B. 

\begin{thmf}
Let $X$ be a normal complex projective variety with klt singularities and $K_X \equiv 0$.
Suppose that $\pi_1^{\textup{\'et}}\big(X_{\textup{reg}}\big)=\{1\}$.
Suppose furthermore that there exists a projective morphism with connected fibers $f \colon \sX \to \Delta$ from a 
normal analytic space $\sX$ whose canonical divisor $K_\sX$ is $\mathbb{Q}$-Cartier
onto the complex open unit disk $\Delta$ such that 
$X \cong f^{-1}(0)$ and such that $f^{-1}(t)$ is smooth for $t \neq 0$.
Then there exists a decomposition
of $X$ into a product of irreducible, Calabi-Yau and symplectic varieties.
\end{thmf}

\subsection*{Structure of the paper}Section \ref{sec0} is mainly devoted to setting up the basic notation. We have also gathered a number of facts and basic results which will later be used in the proofs. 
Sections \ref{sec1}, \ref{sec2} and \ref{sec3} consist of technical preparations.
In Section \ref{sec1}, we recall some results on deformations of K\"ahler-Einstein metrics on smoothable singular spaces with numerically trivial canonical class.
In Section \ref{sec2}, we establish a structure result for families of mildly singular varieties with trivial canonical class.
In Section \ref{sec3}, we use results from \cite{GGK} to analyze the stability of the tangent sheaf of smoothable
singular spaces with numerically trivial canonical class.
Section \ref{sec4} is devoted to the proof of Theorem~B.
In Section \ref{sec5}, we prove Theorem~A.
Finally, in Section \ref{sec6}, we give examples of smoothable (irreducible) Calabi-Yau and symplectic varieties. We have also
collected examples which illustrate to what extent our results are sharp.

\subsection*{Acknowledgements} 
The project started while the authors were visiting the Freiburg Institute for Advanced Studies. The authors would like to thank this institution for its support and Stefan Kebekus for the invitation. 
They would like to thank Beno\^it Claudon, Daniel Greb, Stefan Kebekus, Mihai P\u{a}un, and Song Sun for interesting discussions concerning the content of this paper.

The first author was partially supported by the project Foliage of Agence Nationale de la Recherche, under agreement ANR-16-CE40-0008-01.

The second author is partially supported by NSF Grant DMS-1510214.

\section{Notation, conventions, and basic facts}
\label{sec0}

\subsection{Global Convention}
Throughout the paper we work over the field $\mathbb{C}$ of complex numbers.

A \textit{variety} is a reduced and irreducible scheme of finite type over $\mathbb{C}$. An \textit{analytic variety} is a reduced and irreducible analytic space. Given a scheme $X$, we denote by $X^{\textup{an}}$ the associated analytic space, equipped with the Euclidean topology.   

Given a scheme or an analytic space $X$, we denote by $X_{\textup{reg}}$ its smooth locus.  

\subsection{Reflexive differential forms}
Given a normal (analytic) variety $X$, we denote the sheaf of K\"{a}hler differentials by
$\Omega^1_X$. 
If $0 \le p \le \dim X$ is any number, write
$\Omega_X^{[p]}:=(\Omega_X^p)^{**}$.
The tangent sheaf will be denoted by $T_X:=(\Omega_X^1)^*$.

\subsection{Quasi-\'etale covers}

\begin{defn}
A morphism $\gamma\colon Y \to X$ between normal (analytic) varieties is called a \textit{quasi-\'etale cover}
if $\gamma$ is finite and \'etale in codimension one.
\end{defn}

\begin{rem}
Let $\gamma\colon Y \to X$ be a quasi-\'etale cover. By the Nagata-Zariski purity theorem, $\gamma$ branches only on the singular set of $X$. In particular, we have $\gamma^{-1}(X_{\textup{reg}}) \subset Y_{\textup{reg}}$. 
\end{rem}

\subsection{Stability}
The word ``stable'' will always mean ``slope-stable with respect to a
given polarization''.

\begin{defn}[{\cite[Definition 7.2]{gkp_bo_bo}}]\label{def:strongStab}
Let $X$ be a normal projective (analytic) variety of dimension $n$, and let $\sG$ be a reflexive
coherent sheaf. We call $\sG$ \emph{strongly stable}, if
for any quasi-\'etale cover $\gamma \colon Y \to X$, and for any choice of ample divisors $H_1, \ldots, H_{n-1}$ on $Y$, the reflexive pull-back $\gamma^{[*]} \sG$ is stable with respect to $(H_1, \ldots, H_{n-1})$.
\end{defn}
 
\subsection{Smoothings}
We will use the following notation.

\begin{notation}
Let $f \colon \sX \to T$ be a morphism (resp. holomorphic map) of schemes (resp. analytic spaces). We will denote by $\sX_t$ the fiber of $f$ over $t \in T$. 
\end{notation}

\begin{defn}
Let $X$ be a compact analytic space. A \textit{smoothing of $X$} is a flat proper holomorphic map $f \colon \sX \to \Delta$, where $\sX$ is an analytic space and $\Delta$ is the complex open unit disk, such that 
$\sX_0 \cong X$ and $\sX_t$ is smooth for any $t \neq 0$.
A \textit{smoothing of a proper scheme $X$} is a smoothing of the associated analytic space $X^{\textup{an}}$.
Let $f \colon \sX \to \Delta$ be a smoothing of a compact analytic space or a proper scheme. 
We say that $f$ is a \textit{projective smoothing} if $f$ is a projective map. 
If $\sX$ is normal, then we say that
$f$ is a \textit{$\mathbb{Q}$-Gorenstein smoothing} if $K_{\sX/\Delta}$ is $\mathbb{Q}$-Cartier.

Let $X$ be a proper scheme. A \textit{smoothing of $X$ over an algebraic curve} is a flat proper
morphism $f \colon \sX \to C$, where $\sX$ is a scheme and $C$ is a connected algebraic curve, such that 
$\sX_{t_0} \cong X$ for some point $t_0$ on $C$ and $\sX_t$ is smooth for any $t \neq t_0$.
A smoothing $f \colon \sX \to C$ of $X$ over an algebraic curve is said to be \textit{projective} if $f$ is a projective morphism. If $\sX$ is normal, then we say that
$f$ is a \textit{$\mathbb{Q}$-Gorenstein smoothing} if $K_{\sX/C}$ is $\mathbb{Q}$-Cartier.
\end{defn}

The following elementary facts will be used throughout.

\begin{fact}\label{remark:normality}
Let $\sX$ be an analytic space (resp. a scheme), and let $f \colon \sX \to T$ be a flat holomorphic map (resp. morphism)
with connected fibers
onto a smooth connected analytic (resp. algebraic) curve $T$. If $\sX_t$ is normal for any $t \in T$, then so is $\sX$ by \cite[Corollaire 5.12.7]{ega24} and \cite[Th\'eor\`eme 1.2]{dieudonne_grothendieck}. In particular, $\sX$ is reduced and locally irreducible (resp. reduced and irreducible).
\end{fact}

\begin{fact}
Let $\sX$ be a variety, and let $X \subset \sX$ be a Cartier divisor. If $X$ is regular at $x$, then so is $\sX$.
\end{fact}

\begin{fact}\label{remark:inversion_adjunction}
In the setup of Fact \ref{remark:normality}, suppose that $\sX$ is a normal variety, and
that $K_\sX$ is $\mathbb{Q}$-Cartier. If $\sX_t$ has klt singularities for some point $t$ on $C$, then $\sX$ has klt singularities in a neighborhood of $\sX_t$ by inversion of adjunction (see \cite[Corollary 7.6]{kollar97}). This implies that $\sX$ has rational singularities in a neighborhood of $\sX_t$ (see \cite[Theorem 1.3.6]{kmm}).
\end{fact}

\begin{rem}\label{remark:lc_canonical}
In the setup of Fact \ref{remark:normality}, suppose that $\sX$ is a normal variety, and
that $K_\sX$ is $\mathbb{Q}$-Cartier.
If $\sX_{t_0}$ has lc singularities for some point $t_0$ on $C$ and $\sX_t$ has canonical singularities for $t \neq t_0$, then $\sX \times_C C_1$ has canonical singularities for any finite morphism $C_1 \to C$ from a smooth algebraic curve $C_1$ by \cite[Theorem 2.5]{karu_stable}. Note that the proof of \cite[Theorem 2.5]{karu_stable} relies on \cite[Theorem 1.2]{bchm}. However, we will not need this stronger statement.
\end{rem}

Let $X$ be a normal projective variety. If $X$ admits a projective $\mathbb{Q}$-Gorenstein
smoothing over an algebraic curve then $X$ obviously admits a projective $\mathbb{Q}$-Gorenstein smoothing. The main result of the present section is a partial converse to this observation.
See Lemma \ref{lemma:smoothing_disk_versus_alg_curve} and Proposition \ref{proposition:smoothing_disk_versus_alg_curve} for precise statements.

\begin{lemma}\label{lemma:smoothing_disk_versus_alg_curve}
Let $X$ be a normal projective variety with Gorenstein singularities. If $X$ admits a projective smoothing, then $X$ admits a projective $\mathbb{Q}$-Gorenstein smoothing over an algebraic curve.
\end{lemma}

\begin{proof}
Let $f\colon \sX \to \Delta$ be a projective smoothing of $X$. We may assume without loss of generality that $\sX \subset \mathbb{P}^N\times \Delta$ for some positive integer $N$. Let $\Delta \to \textup{Hilb}(\mathbb{P}^N)^{\textup{an}}$ be the universal holomorphic map, and let $H$ any irreducible component of $\textup{Hilb}(\mathbb{P}^N)$ such that $H^{\textup{an}}$ contains the image of $\Delta$. 
Let also $U \subset \mathbb{P}^N \times H$ be the universal family.
Note that $\dim H \ge 1$ and that the generic fiber of the natural morphism $U \to H$ is smooth. Replacing $H$ by a Zariski open neighborhood of $[X]$, if necessary, we may assume that 
for any $t \in H$, $U_t$ is normal and Cohen-Macaulay by \cite[Th\'eor\`emes 12.2.1 et 12.2.4]{ega28}.
Denote by $\omega_{U/H}$
the dualizing sheaf of the Cohen-Macaulay morphism $U \to H$ (see \cite[3.5]{conrad}).
By \cite[Theorem 3.5.6]{conrad}, for any closed point $t \in H$, we have  
${\omega_{U/H}}_{|U_t}\cong \omega_{U_t}$ where $\omega_{U_t}$ is the dualizing sheaf of $U_t \to \textup{pt}$.
Set $t_0:=[X] \in H$.
Because $\omega_{U_{t_0}}$ is invertible by assumption, we see that $\omega_{U/H}$ is invertible in a Zariski open neighborhood of $X$. By shrinking $H$ if necessary, we may therefore assume that $\omega_{U/H}$ is invertible.
Let $C \to H$ be a smooth curve passing through $[X]$ and a general point, and set $\sY:=C\times_H U$.
It comes with a natural morphism $g \colon \sY \to C$. 
Note that $g$ is a projective smoothing of $X$ over $C$. By Fact \ref{remark:normality}, we know that
$\sY$ is normal.
Applying \cite[Theorem 3.5.6]{conrad} again, we see 
that $\omega_{\sY/C}$ is invertible. Moreover, its restriction to the locus 
${\sY}_{\textup{reg}}^\circ\subset {\sY}_{\textup{reg}}$
where $g$ is smooth is $\omega_{{\sY}_{\textup{reg}}^\circ/H}$. On the other hand,
$\omega_{{\sY}_{\textup{reg}}^\circ/H} \cong \sO_{{\sY}_{\textup{reg}}^\circ}\big(K_{{\sY}_{\textup{reg}}^\circ/H}\big)
\cong {\sO_{\sY}\big(K_{\sY/H}\big)}_{|{\sY}_{\textup{reg}}^\circ}$.
It follows that
$\omega_{\sY/H} \cong \sO_{\sY}\big(K_{\sY/H}\big)$ since both are reflexive sheaves. Hence
$K_{\sY/H}$ is Cartier, completing the proof of the lemma.
\end{proof}

The same argument used in the proof of Lemma \ref{lemma:smoothing_disk_versus_alg_curve} above shows that the following holds.

\begin{lemma}\label{lemma:smoothing_disk_versus_alg_curve2}
Let $X$ be a normal projective variety. If $X$ admits a projective smoothing, then $X$ admits a projective smoothing over an algebraic curve.
\end{lemma}

The proof of \cite[Proposition 1.4.14]{lazarsfeld1} apply in the analytic setting to show that the following holds.

\begin{lemma}\label{lemma:locus_nef}
Let $f \colon \sX \to T$ be a projective holomorphic map (resp. morphism) of analytic spaces (resp. schemes), and let $\sL$ be a line bundle
on $\sX$. Suppose furthermore that $f$ is surjective.
Then the set of points 
$t$ on $T$ such that $\sL_{|\sX_{t}}$ is not nef is a countable union of analytic subsets.
\end{lemma}

We end the preparation for the proof of Proposition \ref{proposition:smoothing_disk_versus_alg_curve} with the following observation.

\begin{lemma}\label{lemma:num_trivial_versus_torsion}
Let $f \colon \sX \to T$ be a flat projective holomorphic map (resp. morphism) 
of analytic spaces (resp. schemes). Suppose that $\sX$ is normal, $T$ is smooth, and that
$\sX_t$ is connected with klt singularities for any point $t \in T$. Suppose furthermore that $K_{\sX/T}$ is $\mathbb{Q}$-Cartier. If $K_{\sX_{t_0}} \equiv 0$ for some point $t_0\in T$, then
there exists a Zariski open cover $(T^\alpha)_{\alpha \in A}$ of $T$ such that $K_{\sX^\alpha/T^\alpha}$ is torsion where $\sX^\alpha:=f^{-1}(T^\alpha)$. In particular, 
$K_{\sX_t}$ is torsion for all $t \in T$.
\end{lemma}

\begin{proof}
Let $m_0$ be a positive integer such that $m_0K_{\sX/T}$ is a Cartier divisor.
Applying Lemma \ref{lemma:locus_nef} in the analytic setting or \cite[Proposition 1.4.14]{lazarsfeld1}
in the algebraic setting
to $\pm m_0K_{\sX/T}$, we see that 
the set of points $t \in T$ such that $K_{\sX_t}\not\equiv 0$ is a countable union of proper Zariski closed subsets.
If $K_{\sX_t}\equiv 0$, then $K_{\sX_t}$ is torsion by \cite[Corollary V 4.9]{nakayama04}. Note also that
we have $m_0K_{\sX_t}\sim_\mathbb{Z} {m_0K_{\sX/T}}_{|\sX_t}$ by the adjunction formula.
Let now $m$ be a positive integer. Because the functions 
$t \mapsto h^{0}\big(\sX_t,\sO_{\sX_t}(\pm mm_0K_{\sX_t})\big)$ are upper semicontinuous in the Zariski 
topology on $T$ (see \cite[Chapter III, Theorem 4.12]{banica} in the analytic setting or \cite[Theorem 12.8]{hartshorne77} in the algebraic setting), the set of points $t$ on $T$ such that 
$mm_0K_{\sX_t} \sim_\mathbb{Z} 0$ is closed.
It follows that 
there exists a positive integer $m_1$ such that 
$m_1m_0K_{\sX_t} \sim_\mathbb{Z} 0$ for all $t \in T$.
From \cite[Chapter III, Theorem 4.12]{banica} in the analytic setting or \cite[Corollary 12.9]{hartshorne77} in the algebraic setting, we see that
$f_*\sO_{\sX}(m_1m_0K_{\sX/T})$ is a line bundle. 
Let $(T^\alpha)_{\alpha \in A}$ be a Zariski open affine cover of $T$ such that
${f_*\sO_{\sX}(m_1m_0K_{\sX/T})}_{|T^\alpha}\cong \sO_{T^\alpha}$ for all $\alpha \in A$. Set $\sX^\alpha:=f^{-1}(T^\alpha)$, and let $t \in T^\alpha$.
Because the formation of $f_*\sO_{\sX}(m_1m_0K_{\sX/T})$
commutes with base change (see \textit{loc. cit.}), any non-zero section of $m_1m_0{K_{\sX/T}}_{|\sX_t}$ extends to a global section of $m_1m_0{K_{\sX/T}}_{|\sX^\alpha}$ that is nowhere vanishing in a neighborhood of $\sX_t$. Refining the cover, if necessary, we conclude that $m_1m_0{K_{\sX/T}}_{|\sX^\alpha} \cong \sO_{\sX^\alpha}$. This completes the proof of the lemma.
\end{proof}

\begin{prop}\label{proposition:smoothing_disk_versus_alg_curve}
Let $X$ be a normal projective variety with klt singularities. Suppose that $K_X \equiv 0$. 
If $X$ admits a projective $\mathbb{Q}$-Gorenstein smoothing, then 
there exist a normal projective variety $Y$, a quasi-\'etale cover $Y \to X$, and a projective smoothing $\sY \to C$ of $Y$ 
over an algebraic curve such that $K_{\sY/C}\sim_\mathbb{Z} 0$.
\end{prop}

\begin{proof}
Let $f\colon \sX \to \Delta$ be a projective smoothing of $X$.
Applying Lemma \ref{lemma:num_trivial_versus_torsion}, we see that
$K_{\sX/\Delta}$ is torsion.

By \cite[Definition 2.52 and Lemma 2.53]{kollar_mori}, there exists a normal analytic variety $\sY$ and a
finite cover $\gamma \colon \sY \to \sX$, \'etale over $\sX_{\textup{reg}}$,
such that $K_{\sY/\Delta}\sim_\mathbb{Z} 0$. 
Note that $\gamma_t\colon\sY_t \to \sX_t$ is \'etale for any $t \neq 0$. In particular, $\sY_t$ is smooth if $t \neq 0$. Note also that
$\sY_0$ is normal, and that $\gamma_0\colon \sY_0 \to \sX_0$ is a quasi-\'etale cover.
Applying \cite[Proposition 3.16]{kollar97} to $\gamma_0$, we see that
$Y:=\sY_0$ has klt singularities. Moreover, by the adjunction formula, we have $K_{Y}\sim_\mathbb{Z} 0$.

By Lemma \ref{lemma:smoothing_disk_versus_alg_curve}, $Y$ admits a projective $\mathbb{Q}$-Gorenstein smoothing over an algebraic curve. Arguing as above, we see that there exist a normal projective variety 
$Y_1$ as well as a quasi-\'etale cover $\gamma_1\colon Y_1 \to Y$, and a projective smoothing $\sY_1 \to C_1$ of $Y_1$ over an algebraic curve such that $K_{\sY_1/C_1}\sim_\mathbb{Z} 0$, completing the proof of the proposition.
\end{proof}

\section{K\"ahler-Einstein metrics on smoothable spaces}
\label{sec1}

In this section, we work in the following setting, referred to later as the analytic setting. 

\subsection{The analytic setting}\label{sec:anset}
Let $f \colon \sX \to \Delta$ be a projective smoothing of 
a normal projective (analytic) variety $X$  
such that $K_{\sX/\Delta} \sim_{\mathbb Z} 0$. Suppose moreover that $X$ has canonical singularities.  
Let $\mathscr L$ be a relatively ample line bundle on $\X$, and set $\L_t:=\L_{|\X_t}$.
Given $t \neq 0$, we denote by $g_t$ the unique Ricci-flat K\"ahler metric on $\sX_t$ 
whose fundamental form $\omega_t$ satisfies $[\omega_t]=[c_1(\sL_t)]\in H^2(\sX_t,\mathbb{R})$.
The existence of $g_t$ is established in \cite{Yau78}. We denote by $g_0$ the Ricci-flat K\"ahler metric on $X_{\reg}$ given by 
\cite[Theorem 7.5]{EGZ} applied to $(\sX_0,\sL_0)$.

\subsection{Smooth convergence on $\Xr$}
In the setting described above, let $\Phi\colon\Xr \times \Delta \to \X$ be a smooth embedding such that $\Phi(x,t)\in \X_t$ for any $(x,t)\in \Xr \times \Delta$ and $\Phi_{|\Xr\times \{0\}}=\mathrm{Id}_{\Xr}$ (see \cite[p. 1547]{RZ0}). Let us write $\Phi_t:=\Phi_{|\Xr\times \{t\}}\colon\Xr \to \X_t$.
If $I_t$ denotes the complex structure on $\X_t$ for $t\neq 0$, and $\Xr$ for $t=0$, then it is not hard to see that $\Phi_{t}^*I_t$ converges to $I_0$ in the $\sC^{\infty}_{\rm loc}(\Xr)$-topology. The following theorem is due to Rong-Zhang (see \cite[Theorem 1.4]{RZ}).

\begin{thm}[\cite{RZ}]
\label{thm:rz}
The family of Riemannian metrics $\Phi_{t}^*g_t$ converges to $g_0$ when $t\to 0$, in the $\sC^{\infty}_{\rm loc}(\Xr)$-topology. 
\end{thm}

\begin{proof}[Sketch of proof]
We recall the main arguments of the proof for the convenience of the reader. 

The first step is to show that $\om_t$ and $\om_{\rm FS,t}$ differ by the $\ddc$ of an uniformly bounded 
potential $\vp_t$, that we can assume to be normalized by $\int_{X_t} \vp_t\om_t^n=0$. Here, $\om_{\rm FS,t}={\om_{\rm FS}}_{|\X_t}$ under the embedding 
$\X \hookrightarrow \P^N \times \Delta$ given by sections of $\L$. This is proved in \cite[Lemma 3.1]{RZ} 
using Moser iteration given that the Sobolev and Poincar\'e constants of $(\X_t,\om_t)$ are bounded, which in 
turn is a consequence of the diameter estimate \cite[Theorems B.1 and 2.1]{RZ}. 

The $L^{\infty}$-estimate on $\vp_t$ combined 
with Chern-Lu inequality implies that $\om_t \ge C^{-1}\om_{\rm FS,t}$ for some uniform $C>0$ (see the 
proof of \cite[Lemma 3.2]{RZ}). This enables to get estimates at any order on $\Phi_t^*\vp_t$ over compact subsets of $\Xr$. Therefore, one can extract smooth sequential limits $\vp_{\infty}$ of $\Phi_t^*\vp_t$ over $\Xr$: $\vp_{\infty}$ is bounded on the whole $\Xr$ and as $\Phi_{t}^*I_t$ converges to $I_0$, $\vp_{\infty}$ satisfies the same Monge-Amp\`ere equation as $\vp_0$ on $\Xr$, where $\vp_0$ is the normalized potential of $\om_0$ with respect to $\om_{FS,0}$. By the choice of the normalization, $\vp_{\infty}$ and $\vp_0$ coincide, and therefore $\Phi_t^*\vp_t$ converges locally smoothly to $\vp_0$, which concludes the proof. 
\end{proof}

\begin{rem}\label{cstr}
An important observation is that if $\X_t$ admits another complex structure $J_t$ compatible with $g_t$ (in the sense that $J_t$ is unitary and parallel with respect to $g_t$), then 
one can extract sequences $t_j\to 0$ such that $\Phi_{t_j}^*J_{t_j}$ converges locally smoothly to a complex structure $J_0$ over $\Xr$ which is compatible with respect to $g_0$. Indeed, $\Phi_{t}^*J_{t}$ is almost $g_0$-unitary and $g_0$-parallel in the sense that the tensors $(\Phi_t^*J_t)^{*_0}(\Phi_t^*J_t)-\mathrm {Id}$ and $\nabla^{g_0}(\Phi_t^*J_t)$ converge to zero on $\Xr$, where $*_0$ denotes the adjoint with respect to $g_0$. The first property gives order $0$ estimates while the second one enables to get higher order ones. Arzela-Ascoli theorem combined with a diagonal argument yield the result. 
\end{rem}

\subsection{Identification of the Gromov-Hausdorff limit}
One can actually understand the global limit of $(\X_t,g_t)$, but this requires considerable more work and will not be used in the following. The next result is essentially 
contained in \cite[Section 3]{SSY} (see also \cite[Section 4.2]{LWX}). As the context is slightly different here since we deal with Calabi-Yau manifolds rather than log-Fano manifolds, we will recall their main arguments and point out the slight changes to operate.

\begin{thm}[\cite{SSY,LWX}]
\label{thm:gh}
In the standard setting, the metric spaces $(\X_t, \om_t)$ converge in the Gromov-Hausdorff sense to 
$(X,d)$ when $t\to 0$ for some metric $d$ on $X$. Moreover, 
the convergence is smooth on $X_{\reg}$, and the restriction of $d$ to $X_{\reg}$ is induced by the Riemannian metric $g_0$.
\end{thm}

\begin{proof}[Sketch of proof]
The fundamental results of \cite{DS} give the convergence in the sense of the statement above to a 
projective variety endowed with a singular K\"ahler-Einstein metric, but the key point is to identify 
that variety with the central fiber $X$ of our family.

As in the proof of Theorem~\ref{thm:rz}, one can write  $\om_t=\om_{\rm FS,t}+\ddc \vp_t$ with $||\vp_t||_{L^{\infty}}\le C$, $C$ being independent of $t\in \Delta^*$. This has been seen to imply the estimate $\om_t \ge C^{-1}\om_{\rm FS,t}$. Also, it is easy but important to notice that 
the bound on the potential remains valid after Veronese re-embeddings (see \cite[Lemma 3.3]{SSY}).

The next step consists in comparing the two different embeddings of $\X_t$ into a large $\P^{M}$, one 
being given by a Veronese embedding (using a large power of $\L$, say $k$) and the other one being an 
embedding by $L^2$ sections of $\L_t^{\otimes k}$ with respect to $\om_t$. In the following, we may assume that the 
first re-embedding is the identity while the second one will be denoted by $i_t$. So in $\P^{M}$, we have 
two isomorphic varieties, namely $\X_t$ and $i_t(\X_t)$. The variety $i_t(\X_t)$ can be obtained from $\X_t$ 
by a transformation $g_t\in \mathrm{PGL}(\mathbb C^{M+1})$ that sends a basis $(s_{0,t}, \ldots, s_{M,t})$ 
of $H^0\big(\X_t,\L_t^{\otimes k}\big)$ to a $L^2$-orthonormal basis $(\sigma_{0,t}, \ldots, \sigma_{M,t})$, where the first basis of sections is obtained by the standard embedding induced by $\L^{\otimes k}$. Then one can prove using the estimates 
on $||\vp_t||_{L^{\infty}},\int_{X_t}\om_t^n, \tr_{\om_t}(\om_{\rm FS,t})$ that $g_t$ evolves in a 
compact subset of $\mathrm{PGL}(\mathbb C^{M+1})$ (see proof of \cite[Theorem 3.1]{SSY}).

The last step invokes the main result of \cite{DS} that guarantees in this context that one can choose $k$ 
and a sequence $t_j \to 0$ such that the re-embedding $i_{t_j}(X_{t_j})$ converges both in the sense of 
cycle and in the Gromov-Hausdorff sense to a projective variety $W$ (for some metric on $W$). Up to extract 
a subsequence, one can assume that $g_{t_j}$ converges to $g\in \mathrm{PGL}(\mathbb C^{M+1})$ so 
that $W=g(X)$ as projective varieties. The rest is a consequence of \cite{DS}.
\end{proof}

\section{Relative Albanese morphism}
\label{sec2}

The Albanese morphism, that is, the universal morphism to an abelian variety (see \cite{serre}), is one important tool in the study of varieties with trivial canonical divisor. If $X$ is a projective variety with rational singularities, recall from \cite[Lemma 8.1]{kawamata85} that
$\Pic^\circ(X)$ is an abelian variety, and that $\Alb(X)\cong\big(\Pic^\circ(X)\big)^\vee$. 
Moreover, the Albanese morphism $a_X \colon X \to \Alb(X)$ is induced by the universal line bundle.
In particular, $\dim \Alb(X) = h^1(X,\sO_X)$. 
The following result of Kawamata describes the Albanese map of $X$ (see also \cite[Section 3]{brion_action}). 

\begin{prop}[{\cite[Proposition 8.3]{kawamata85}}]\label{proposition:kawamata}
Let $X$ be a normal projective variety $X$ with
canonical singularities. Assume that $K_X$ is numerically trivial. Then $K_X$
is torsion, the Albanese map $a_X \colon X \to \Alb(X)$ is surjective and has the structure of an \'etale-trivial fiber bundle.
More precisely, the following holds. The neutral component $A$ of the automorphism group $\textup{Aut}(X)$ of $X$ is an abelian variety. It acts on $\Alb(X)$ by translation, compatibly with its action on $X$. The induced morphism $A \to \Alb(X)$ is an isogeny and the fiber product over $\Alb(X)$ decomposes as a product 
$$X \times_{\Alb(X)} A \cong F \times A,$$
where $F$ is a fiber of $a_X$.
\end{prop}

In this section we extend Proposition \ref{proposition:kawamata} to the relative setting.

\begin{prop}\label{proposition:Abelian_factor}
Let $\sX$ be a normal variety, and let
$f \colon \sX \to C$ be a flat projective morphism with connected fibers onto a smooth connected algebraic curve $C$.
Suppose that $\sX_t$ has klt singularities for any point $t \in C$.
Suppose furthermore that $K_{\sX/C}\sim_\mathbb{Z} 0$.
\begin{enumerate}
\item The neutral component $\textup{Aut}^0(\sX_t)$ of the automorphism group of $\sX_t$ is an abelian variety, and 
the algebraic groups $\textup{Aut}^0(\sX_t)$ fit together to form an abelian scheme $\sA$ over $C$.
\item Suppose that $f$ has a section, and consider the morphism 
$$a_{\sX/C}\colon \sX \to \big(\Pic^\circ(\sX/C)\big)^\vee$$
induced by the universal line bundle. 
Let 
$0\colon C \to \big(\Pic^\circ(\sX/C)\big)^\vee$ denotes the neutral section, and 
set $\sY=a_{\sX/C}^{-1}(0)$.
Then $\sY \to C$ is a flat projective morphism with normal connected fibers, and 
there exists a relative isogeny $\sA \to \big(\Pic^\circ(\sX/C)\big)^\vee$ such that
$$\sX \times_{\big(\Pic^\circ(\sX/C)\big)^\vee}\sA \cong \sY\times_{C} \sA.$$
\end{enumerate}
\end{prop}

Before we give the proof of Proposition \ref{proposition:Abelian_factor}, we need the following auxiliary result.

\begin{lemma}\label{lemma:continuity}
Let $f \colon \sX \to C$ be a flat projective morphism with connected fibers from a 
normal variety $\sX$ onto a smooth connected algebraic curve $C$.
Suppose that $\sX$ has rational singularities. Then $h^i\big(\sX_t,\sO_{\sX_t}\big)$ is independent of $t \in C$ for any integer $i \ge 0$.
\end{lemma}

\begin{proof}
Let $\wb C$ be a smooth compactification of $C$, and let $\wb \sX$ be a projective compactification of $\sX$ such that $f$ extends to a morphism $\wb f \colon \wb \sX \to \wb C$. Finally, let $\nu \colon \wh \sX \to \wb \sX$ be a desingularization of $\wb \sX$. Set $\wh f := \wb f \circ \nu$. It follows from \cite[Corollary 3.9]{kollar_higher2} that the sheaves $R^i\wh f_*\sO_{\wh \sX}$ are torsion free, and hence locally free. The relative Leray spectral sequence 
gives ${R^i\wh f_*\sO_{\wh \sX}}_{|C}\cong R^i f_*\sO_{\sX}$ using
the assumption that $\sX$ has rational singularities.
This in turn implies that $h^i\big(\sX_t,\sO_{\sX_t}\big)$ is independent of $t \in C$
by a theorem of Grothendieck (see \cite[Theorem 12.11]{hartshorne77}).
\end{proof}

\begin{proof}[Proof of Proposition \ref{proposition:Abelian_factor}]
Note that ${\sX}_t$ has canonical singularities since $K_{{\sX}_t}\sim_\mathbb{Z} 0$ is a Cartier divisor by the adjunction formula.
Write $n:=\dim \sX - 1$. 
From Lemma \ref{lemma:continuity} and Fact \ref{remark:inversion_adjunction}, we see that
$h^{n-1}\big(\sX_t,\sO_{\sX_t}\big)$ is independent of $t \in C$. On the other hand, we have 
$h^{n-1}\big(\sX_t,\sO_{\sX_t}\big)=h^0\big(\sX_t,\Omega^{[n-1]}_{\sX_t}\big)$ by 
\cite[Proposition 6.9]{gkp_bo_bo}. It follows that 
$h^{0}\big(\sX_t,T_{\sX_t}\big)$ is independent of $t \in C$ since 
$\Omega^{[n-1]}_{\sX_t}\cong T_{\sX_t}$. Note that $h^{n-1}\big(\sX_t,\sO_{\sX_t}\big)=h^{1}\big(\sX_t,\sO_{\sX_t}\big)$ by Serre duality, so that $h^0(\sX_t,T_{\sX_t})=h^{1}\big(\sX_t,\sO_{\sX_t}\big)$.
By
\cite[Expos\'e VI$_{\textup{B}}$, Proposition 1.6]{sga3}, the group scheme $\textup{Aut}(\sX/C)$ is then
smooth over $C$. 
Note that the existence of $\textup{Aut}(\sX/C)$ is guarantee by \cite{fga221}.
Let $t\in C$, and denote by $\textup{Aut}^0(\sX_t)$ the neutral component of 
the automorphism group $\textup{Aut}(\sX_t)$ of $\sX_t$. 
We have $\dim \textup{Aut}^0(\sX_t)=h^0(\sX_t,T_{\sX_t})=h^{1}\big(\sX_t,\sO_{\sX_t}\big)$.
We claim that $\textup{Aut}^0(\sX_t)$ is an abelian variety. If not, by a theorem of Chevalley, 
$\textup{Aut}^0(\sX_t)$ contains an algebraic subgroup isomorphic either to $\mathbb{G}_a$ or to $\mathbb{G}_m$,
and hence, $X$ is uniruled. On the other hand, $\kappa(\sX_t)=0$ since $K_{\sX_t} \sim_\mathbb{Z} 0$ and
$\sX_t$ has canonical singularities, yielding a contradiction. 
Now, recall from \cite[Expos\'e VI$_{\textup{B}}$, Th\'eor\`eme 3.10]{sga3} 
that the algebraic groups
$\textup{Aut}^0(\sX_t)$ fit together to form a group scheme $\sA$
over $C$. Note that $\sA$ is quasi-projective over $C$ (see \cite{fga221}). Applying \cite[Corollaire 15.7.11]{ega28}, we see that $\sA$ is proper, and hence projective over $C$. In particular, $\sA$ is an abelian scheme over $C$.

\medskip

Suppose from now on that $f$ has a section $s\colon C \to \sX$, and consider the relative Picard scheme
$\Pic(\sX/C)$ whose existence is guaranteed by \cite[Th\'eor\`eme 3.1]{fga232}.
By \cite[Lemma 8.1]{kawamata85}, $\Pic^\circ(\sX_t)$ is an abelian variety for any point $t$ on $C$.
As above, the algebraic groups $\Pic^\circ(\sX_t)$ fit together to form a group scheme 
$\Pic^\circ(\sX/C)$ over $C$. Note that $\Pic^\circ(\sX/C) \subset \Pic(\sX/C)$ is an open subscheme, quasi-projective over $C$ by \cite[Theorem 5]{neron_models}. Using \cite[Corollaire 15.7.11]{ega28} again, we conclude that it is projective over $C$. Let $\big(\Pic^\circ(\sX/C)\big)^\vee$ be the dual abelian scheme (see \cite[Corollary 6.8]{git}), and consider the morphism 
$$a_{\sX/C}\colon \sX \to \big(\Pic^\circ(\sX/C)\big)^\vee$$
induced by the universal line bundle.
Consider also the action $\phi\colon \sA\times_C\sX \to \sX$, and the second projection $p_2\colon \sA\times_C\sX \to \sX$. By the rigidity lemma (\cite[Proposition 6.1]{git}), there exists a morphism
$$\psi \colon \sA \times_C \big(\Pic^\circ(\sX/C)\big)^\vee \to \big(\Pic^\circ(\sX/C)\big)^\vee$$
such that the square

\centerline{
\xymatrix{
\sA\times_C \sX \ar[rr]^{\phi}\ar[d]_{(\textup{Id}\times a_{\sX/C})} && \sX \ar[d]^{a_{\sX/C}} \\
\sA \times_C \big(\Pic^\circ(\sX/C)\big)^\vee \ar[rr]_{\psi} && \big(\Pic^\circ(\sX/C)\big)^\vee 
}
}
\noindent is commutative. It follows that $\sA$ acts on $\big(\Pic^\circ(\sX/C)\big)^\vee$ over $C$. By Proposition \ref{proposition:kawamata}, the induced morphism $$\sA \to \big(\Pic^\circ(\sX/C)\big)^\vee$$ is a relative
isogeny. Let $0\colon C \to \big(\Pic^\circ(\sX/C)\big)^\vee$ denotes the neutral section, and set $\sY:=a_{\sX/C}^{-1}(0)$.
Using Proposition \ref{proposition:kawamata} again, one readily checks that 
$$\sX \times_{\big(\Pic^\circ(\sX/C)\big)^\vee}\sA \cong \sY\times_{C} \sA.$$
The proposition then follows easily.
\end{proof}

We end this section with an extension result for differential forms.
We feel that this result might be of independent interest.

\begin{prop}\label{prop:extension_differential_forms}
Let $\sX$ be a normal variety with rational singularities, and let 
$f \colon \sX \to C$ be a flat projective morphism with connected normal fibers onto a smooth connected algebraic curve $C$. Let $0\le p\le \dim \sX - 1$ be an integer, and assume that 
$h^0\big(\sX_t,\Omega_{\sX_t}^{[p]}\big)=h^p(\sX_{t},\sO_{\sX_{t}})$ for any point $t$ on $C$. 
Given $t_0 \in C$ and $\omega \in H^0\big(\sX_{t_0},\Omega_{\sX_{t_0}}^{[p]}\big)$, and
replacing $C$ by an open neighborhood of $t_0$ if necessary, there exists 
$\Omega \in H^0\big(\sX,\Omega_{\sX/C}^{[p]}\big)$
such that $\Omega_{|(\sX_{t_0})_{\textup{reg}}}=\omega_{|(\sX_{t_0})_{\textup{reg}}}$.
\end{prop}

\begin{rem}
\label{rem:hodge}
In the setup of Proposition \ref{prop:extension_differential_forms} above, suppose moreover that $\sX_t$ has klt singularities for any $t\in C$. Then Hodge symmetry 
$h^0\big(\sX_t,\Omega_{\sX_t}^{[p]}\big)=h^p(\sX_{t},\sO_{\sX_{t}})$ holds by \cite[Proposition 6.9]{gkp_bo_bo}.
\end{rem}

\begin{proof}[Proof of Proposition \ref{prop:extension_differential_forms}]
Denote by $i \colon \sX_{\textup{reg}} \into \sX$ the natural morphism, 
so that $\Omega^{[p]}_{\sX/C}\cong i_*\Omega_{\sX_{\textup{reg}/C}}^{p}$ by \cite[Proposition 1.6]{hartshorne80}. 
The sheaf $\Omega^{[p]}_{\sX/C}$ is torsion free, and hence flat over $C$. This implies in particular that
$f_*\Omega^{[p]}_{\sX/C}$ is locally free
of rank 
$h^0\Big(\sX_{t_1},{\Omega^{[p]}_{\sX/C}}_{|\sX_{t_1}}\Big)$ where $t_1 \in C$ is a general point.
On the other hand, if $t_1$ is general enough, then the sheaf ${\Omega^{[p]}_{\sX/C}}_{|\sX_{t_1}}$ is reflexive, and hence 
${\Omega^{[p]}_{\sX/C}}_{|\sX_{t_1}} \cong \Omega^{[p]}_{\sX_{t_1}}$.

Given $t\in C$, consider the exact sequence
$$0 \to \sO_{\sX_{\textup{reg}}}(-\sX_{t}\cap \sX_{\textup{reg}}) \to \sO_{\sX_{\textup{reg}}}\to 
\sO_{\sX_{t}\cap \sX_{\textup{reg}}} \to 0,$$ and
let $j_t \colon \sX_t \cap\sX_{\textup{reg}} \into \sX_t$ denotes the natural morphism.
Tensoring the above exact sequence with the sheaf $\Omega_{\sX_{\textup{reg}/C}}^p$ and applying $i_*$ yield an exact sequence
$$0 \to \Omega^{[p]}_{\sX/C} \otimes \sO_{\sX}(-\sX_t) \to \Omega^{[p]}_{\sX/C} \to 
(j_t)_*\Big({\Omega^{p}_{\sX_{\textup{reg}}/C}}_{|\sX_{t}\cap \sX_{\textup{reg}}}\Big).$$
Now, by \cite[Expos\'e VIII, Corollaire 2.3]{sga2}, 
$(j_t)_*\Big({\Omega^{p}_{\sX_{\textup{reg}}/C}}_{|\sX_{t}\cap \sX_{\textup{reg}}}\Big)$ is a coherent sheaf of $\sO_{\sX_t}$-modules. On the other hand, the complement of $\sX_t \cap\sX_{\textup{reg}}$ in $\sX_t$ has 
codimension at least $2$, and ${\Omega^{p}_{\sX_{\textup{reg}}/C}}_{|\sX_{t}\cap \sX_{\textup{reg}}}$ is locally free. It follows that $(j_t)_*\Big({\Omega^{p}_{\sX_{\textup{reg}}/C}}_{|\sX_{t}\cap \sX_{\textup{reg}}}\Big)$
is reflexive, and thus 
$(j_t)_*\Big({\Omega^{p}_{\sX_{\textup{reg}}/C}}_{|\sX_{t}\cap \sX_{\textup{reg}}}\Big) \cong \Omega_{\sX_t}^{[p]}$ (see \cite[Proposition 1.6]{hartshorne80}).
Thus, we obtain an injective morphism of sheaves
$${\Omega^{[p]}_{\sX/C}}_{|\sX_t} \into \Omega_{\sX_t}^{[p]}$$
and hence
\begin{equation}
\label{eq:equality0}
h^0\Big(\sX_t,{\Omega^{[p]}_{\sX/C}}_{|\sX_t}\Big) \le h^0\big(\sX_t,\Omega_{\sX_t}^{[p]}\big).
\end{equation}
We claim that the inequality above is an equality, that is,
\begin{equation}
\label{eq:equality}
h^0\Big(\sX_t,{\Omega^{[p]}_{\sX/C}}_{|\sX_t}\Big) = h^0\big(\sX_t,\Omega_{\sX_t}^{[p]}\big).
\end{equation}
Indeed, let us start by observing that as the function $t \mapsto h^{0}\Big(\sX_t,{\Omega^{[p]}_{\sX/C}}_{|\sX_t}\big)$ is upper semicontinuous on $C$, we also have for a general point $t_1$ on $C$
$$h^0\Big(\sX_t,{\Omega^{[p]}_{\sX/C}}_{|\sX_t}\Big) \ge 
h^{0}\Big(\sX_{t_1},{\Omega^{[p]}_{\sX/C}}_{|\sX_{t_1}}\Big)=
h^{0}\Big(\sX_{t_1},\Omega^{[p]}_{\sX_{t_1}}\Big).$$ 
Now, recall from Lemma \ref{lemma:continuity} that
$h^p\big(\sX_t,\sO_{\sX_t}\big)$ is independent of $t \in C$.
This implies that $h^{0}\Big(\sX_{t},\Omega^{[p]}_{\sX_{t}}\Big)$
is independent of $t \in C$ since $h^{0}\Big(\sX_{t},\Omega^{[p]}_{\sX_{t}}\Big)=
h^p\big(\sX_{t},\sO_{\sX_{t}}\big)$ by assumption.
It follows that 
$$h^0\Big(\sX_t,{\Omega^{[p]}_{\sX/C}}_{|\sX_t}\Big) \ge h^0\big(\sX_t,\Omega_{\sX_t}^{[p]}\big).$$
Combining this with \eqref{eq:equality0}, we obtain \eqref{eq:equality} hence the claim. Therefore, the quantity $h^0\Big(\sX_t,{\Omega^{[p]}_{\sX/C}}_{|\sX_t}\Big)$ is independent of $t\in C$. By a theorem of Grauert, it follows that the natural map 
$$\big(f_* \Omega^{[p]}_{\sX/C}\big) \otimes \mathbb{C}(t) \to H^0\Big(\sX_t,{\Omega^{[p]}_{\sX/C}}_{|\sX_t}\Big)$$
is an isomorphism (see \cite[Corollary 12.9]{hartshorne77}).
Finally, observe that the natural map ${\Omega^{[p]}_{\sX/C}}_{|\sX_t} \into \Omega_{\sX_t}^{[p]}$
induces an isomorphism
$$H^0\Big(\sX_t,{\Omega^{[p]}_{\sX/C}}_{|\sX_t}\Big) \cong H^0\big(\sX_t,\Omega_{\sX_t}^{[p]}\big)$$
by \eqref{eq:equality}. This ends the proof of the proposition. 
\end{proof}
 
\section{Holonomy and stability of the tangent sheaf}\label{sec3}

In this section, we study the stability of the tangent sheaf of smoothable projective varieties with canonical singularities and trivial canonical divisor. The proof of Proposition \ref{prop:holonomy} relies on 
results proved in \cite{GGK}, which we recall first. 

\subsection{Singular Kähler-Einstein metrics, holonomy and stability}
Let $X$ be a normal projective analytic variety with canonical singularities such that $K_X\sim_{\mathbb Z} 0$, and let $H$ be an ample Cartier divisor. By \cite[Theorem 7.5]{EGZ}, there exists a unique closed positive $(1,1)$-current $\om$ with bounded potentials on $X$ such that $[\om] = [c_1(H)] \in H^2(X,\mathbb{R})$ and such that the restriction of $\omega$ to $X_{\reg}$ is a smooth Kähler form with zero Ricci curvature. Let $g$ be the Riemannian metric associated with $\om_{|X_{\reg}}$ on $X_{\reg}$. 
Given $x\in X_{\reg}$, we view  $(T_x X,g_x)$ as an euclidian vector space.
We denote the holonomy group (resp. restricted holonomy group) of $(X_{\reg},g)$ at $x$ by 
$\Hol(X_{\reg},g)_x$ (resp. $\Hol(X_{\reg},g)^\circ_x$). It comes with a linear representation on $T_xX$. Recall that
$\Hol(X_{\reg},g)_x$ is a subgroup of $\mathrm{SO}(T_xX,g_x)$, and that $\Hol(X_{\reg},g)^\circ_x$ is the connected component of the Lie group $\Hol(X_{\reg},g)_x$. Moreover, the complex structure $I$ on $X_{\reg}$ is parallel with respect to $g$, so that the hermitian metric $h_x$ on $T_x X_{\reg}$ induced by $g_x$ and $I_x$ enables to realize the holonomy group as subgroup of $\mathrm{U}(T_xX,h_x)$. We refer to \cite[Section~2.2]{GGK} for more detailed explanations. 

\medskip

Irreducibility of the holonomy representation and stability of the tangent sheaf are related by the following.

\begin{prop}[{\cite[Proposition 1.5]{GGK}}]
\label{prop:holrep}
In the above setting,
the tangent sheaf $T_X$ is stable \textup{(}resp. strongly stable\textup{)} with respect to any polarization if and only if the holonomy representation $\Hol(X_{\reg},g)_x\circlearrowleft T_xX$ \textup{(}resp. the restricted holonomy representation $\Hol(X_{\reg},g)_x^{\circ}\circlearrowleft T_xX$\textup{)} is irreducible.   
\end{prop}

\subsection{Varieties with strongly stable tangent sheaf} The next result relates varieties strongly stable tangent sheaf to irreducible Calabi-Yau and irreducible symplectic varieties.

\begin{thm}[{\cite[Proposition 1.4]{GGK}}]\label{thm:ggk}
In the setting of Proposition \ref{prop:holrep}, suppose furthermore that $X$ has dimension $n \ge 2$.
If $T_X$ is strongly stable, then one of the following two cases holds. In either case, the action of the restricted
holonomy group on $T_xX$ is isomorphic to the standard representation.
\begin{enumerate}
\item The group $\Hol(X_{\reg},g)_x^{\circ}$ is isomorphic to $\SU(n)$, and $X$ is Calabi-Yau.
\item The dimension of $X$ is even, the group $\Hol(X_{\reg},g)_x^{\circ}$ is isomorphic to $\Sp(\frac{n}{2})$, and there exists a quasi-\'etale cover $Y \to X$ such that $Y$ is irreducible symplectic.
\end{enumerate}
\end{thm}

One of the crucial tools in the proof of Theorem \ref{thm:ggk} above is the so-called \textit{Bochner principle}.
  
\begin{thm}[Bochner principle, {\cite[Theorem A]{GGK}}]\label{holprin} In the setting of Proposition \ref{prop:holrep}, let $p$ and $q$ be non-negative integers, and write $\sE:=\big(\sT_X^{\otimes p}\otimes ({\Omega_X^1})^{\otimes q}\big)^{**}$. Then the restriction to $x$ induces a one-to-one correspondence between global sections of $\sE$ and 
$\Hol(X_{\reg},g)_x$-fixed points in $\sE_x$. 
\end{thm}

\subsection{Smoothings and holonomy}In the analytic setting \ref{sec:anset}, assume that for $t\neq 0$, $\X_t$ is simply-connected and either irreducible Calabi-Yau or irreducible symplectic. The following result then says that $T_X$ is stable. If moreover $\pi_1^{\textup{\'et}}\big(X_{\textup{reg}}\big)=\{1\}$, we conclude that $T_X$ is strongly stable with respect to any polarization. 

\begin{prop}\label{prop:holonomy}
In the analytic setting \ref{sec:anset}, assume that for $t\neq 0$, $\X_t$ is simply-connected and either irreducible Calabi-Yau or irreducible symplectic. We maintain notation of Proposition \ref{prop:holrep}, and set 
$G:=\Hol(X_{\reg},g)_x$, $G^{\circ}:=\Hol(X_{\reg},g)_x^{\circ}$, and $V:=T_xX$.
Then the following holds.
\begin{enumerate}
\item The representation $G\circlearrowleft V$ is irreducible. Equivalently, $T_X$ is stable with respect to any polarization.\label{i}
\item $G\subset \SU(V)$ if $\X_t$ is irreducible Calabi-Yau and $G\subset\Sp(V)$ if $\X_t$ is irreducible symplectic. \label{ii}
\item If $G^\circ$ is trivial, then there exists a quasi-\'etale cover $A \to X$ where $A$ is an abelian variety.\label{iii}
\item If $G^\circ$ is nontrivial, then $G^\circ$ is isomorphic to $\SU(r)\times\cdots\times\SU(r)$ or
$\Sp(r)\times\cdots\times\Sp(r)$, for some positive integer $r$.
In either case, the action of $G^\circ$ on $V$ is isomorphic to the corresponding product of standard representations.
\label{iv}
\end{enumerate}
\end{prop}

\begin{proof}
Let us start with (\ref{ii}). In both cases, $K_X$ is trivial so there exists a non-zero holomorphic $n$-form on $X_{\reg}$. By the Bochner principle (see Theorem \ref{holprin}), this implies that $G\subset \SU(V)$. In the symplectic case, Remark~\ref{cstr} enables to find two compatible complex structures $J_0,K_0$ 
on $(\Xr,g_0)$ satisfying $I_0J_0K_0=-\mathrm{Id}$ and which are smooth limits of such complex structures 
on $(\X_t,g_t)$. In particular, setting $\sigma_0(X,Y):=g_0(J_0X,Y)+ig_0(K_0X,Y)$ defines a \textit{non-degenerate} holomorphic $2$-form $\sigma_0$ on $X_{\reg}$ (see \cite[Proposition 14.15]{Besse}). Therefore, we must have $G\subset \Sp(V)$ by the Bochner principle again.

Moving on to (\ref{i}), assume that the representation $G \circlearrowleft V$ is reducible. Then one can decompose $V=\bigoplus_{i\in I} W_i$ and write $G=\prod_{i\in I} G_i$ where the action of $G_i$ on $W_j$ is irreducible if $j=i$ and trivial otherwise (see \cite[Theorem 10.38]{Besse}). 
Consider the Calabi-Yau case first. We have $G_i \subset \SU(W_i)$ by (\ref{ii}), which by the Bochner principle, provides a non-zero reflexive holomorphic form on $\Xr$ of degree $\dim W_i$. By Proposition~\ref{prop:extension_differential_forms} and Remark~\ref{rem:hodge}, it follows that $W_i$ is either
zero or the whole $V$, which concludes. In the symplectic case, (\ref{ii}) and an elementary computation shows that $W_i$ is even dimensional and that  $G_i \subset \Sp(W_i)$. In particular, one gets $h^0\big(X,\Omega_{X}^{[2]}\big)=\# I$, and Proposition~\ref{prop:extension_differential_forms} enables to conclude once again. Finally, the equivalence with stability is a consequence of Proposition~\ref{prop:holrep} above.

Finally, let us prove (\ref{iii}) and (\ref{iv}). 
Set $V_0:=\{v\in V, \forall g\in G^{\circ}, \, g(v)=v\}$. 
As $G^{\circ}$ is normal in $G$, $V_0$ is fixed by $G$. By (\ref{i}), this implies that $V_0$ is either zero or the whole $V$. Assume $V_0=V$ to start with. Then $G^{\circ}$ is trivial, that is, $(X_{\reg},g)$ is flat. This implies that $\sT_{X_{\reg}}$ is holomorphically flat since the $(1,0)$-part of the Chern connection of $g$ is holomorphic. By \cite[Corollary 1.16]{GKP}, we get the expected quasi-étale cover of $X$ from an abelian variety. So one can now assume that $V_0=\{0\}$. Therefore $G^{\circ} \circlearrowleft V$ decomposes as sum of non-trivial irreducible representations $V=\bigoplus_{j\in J} V_j$ along with $G^{\circ}=\prod_{j\in J} G^{\circ}_j$. By normality of $G^{\circ}$ in $G$, each element $g\in G$ permutes the $V_j$. As the representation $G\circlearrowleft V$ is irreducible, this implies that the representations $G_j^{\circ}\circlearrowleft V_j$ are pairwise isomorphic. Set $r=\dim V_j$.  By the classification of restricted holonomy \cite[Proposition 5.4]{GGK}, we see that $G_j^{\circ}\circlearrowleft V_j$ is  isomorphic to the standard representation of $\SU(r)$ or $\Sp(\frac{r}{2})$, completing the proof of the proposition. 
\end{proof}

\section{Towards a decomposition theorem}
\label{sec4}
\subsection{Reduction to smoothings by simply-connected and irreducible manifolds}
The main result of this section is Proposition \ref{prop:decomposition} below. It reduces the proof of Theorem~B to the case of a smoothing by simply-connected, irreducible Calabi-Yau or symplectic manifolds. 
\begin{prop}\label{prop:decomposition}
Let $X$ be a normal projective variety with klt singularities. Suppose furthermore that $K_X \equiv 0$, and that $X$ admits a projective $\mathbb{Q}$-Gorenstein smoothing. Then there exists an abelian variety $A$
as well as a projective variety $Y$ with canonical
singularities, a quasi-\'etale cover $$A \times Y \to X,$$ and a decomposition
$$Y \cong \prod_{i\in I} Y_i $$
such that the following holds.
\begin{enumerate}
\item The $Y_i$ admit projective $\mathbb{Q}$-Gorenstein smoothings over algebraic curves by irreducible and simply-connected Calabi-Yau, or
symplectic manifolds. 
\item The sheaves $T_{Y_i}$ are slope-stable with respect to any ample polarization on $Y$, with trivial determinants. 
\end{enumerate}
\end{prop}

We first provide technical tools for the proof of our result.

\begin{lemma}\label{lemma:quasi_etale_cover_singularities_fibers}
Let $\sX$ be a normal variety, and let $f \colon \sX \to C$ be a flat projective morphism with connected fibers onto a smooth connected curve $C$. Suppose that $K_{\sX/C}$ is $\mathbb{Q}$-Cartier.
Suppose furthermore that $\sX_t$ has klt singularities for any point $t$ on $C$.
Let $\gamma\colon \sY \to \sX$ be a quasi-\'etale cover with $\sY$ normal.
Then $\sY_t$ is normal with klt singularities for any point $t$ on $C$.
\end{lemma}

\begin{proof}
Note first that $K_{\sY/C}$ is $\mathbb{Q}$-Cartier since we have 
$K_{\sY/C}\sim_\mathbb{Z} \gamma^*K_{\sX/C}$. Moreover, by Fact \ref{remark:inversion_adjunction}, $\sX$ 
has klt singularities. Applying \cite[Proposition 3.16]{kollar97} to $\gamma$, we see that
$\sY$ has klt singularities as well. In particular, $\sY$ is Cohen-Macaulay by \cite[Theorem 1.3.6]{kmm}, and hence so is $\sY_t$ for all $t\in C$. By the Nagata-Zariski purity theorem, $\gamma$ branches only over the 
singular set of $\sX$. On the other hand, we know that the smooth locus of $\sX_t$ is contained in the smooth locus of $\sX$. It follows that $\sY_t$ is smooth in codimension one. Now, from Serre's criterion for normality, we see that $\sY_t$ is normal for any $t\in C$.
Note also that ${\gamma}_t\colon {\sY}_t \to \sX_t$ is a quasi-\'etale cover.
By \cite[Proposition 3.16]{kollar97} applied to ${\gamma}_t$, we conclude that ${\sY}_t$ has klt singularities.
\end{proof}

\begin{rem}
In the setup of Lemma \ref{lemma:quasi_etale_cover_singularities_fibers}, let $\beta \colon B \to C$ be the Stein factorization of $g$. Then $\beta$ is \'etale, and $\gamma$ factors through the \'etale cover 
$B\times_C \sX \to \sX$.
\end{rem}

\begin{lemma}\label{lemma:limit_etale_cover}
Let $f \colon \sX \to C$ be a flat projective morphism with connected fibers from a 
normal variety $\sX$ onto a smooth connected curve $C$.
Suppose that $K_{\sX/C}$ is $\mathbb{Q}$-Cartier.
Suppose furthermore that $\sX_t$ has klt singularities for any point $t$ on $C$.
Let $C^\circ \subset C$ be a dense open set, and let $g^\circ\colon \sY^\circ \to C^\circ$ be a smooth projective morphism with connected fibers. Finally, let $\gamma^\circ\colon \sY^\circ \to f^{-1}(C^\circ)$ be a quasi-\'etale cover such that 
$g^\circ=f_{|f^{-1}(C^\circ)}\circ \gamma^\circ$. Then there exists a 
normal variety $\sY_1$ as well as a
projective morphism $g_1\colon \sY_1 \to C_1$ with connected fibers onto a
smooth curve $C_1$, a finite morphism $\pi\colon C_1 \to C$,
and a quasi-\'etale cover $\gamma_1\colon \sY_1 \to \sX\times_C C_1$ such that
the following holds. Write $C_1^\circ:=\pi^{-1}(C^\circ)$.
\begin{enumerate}
\item We have ${\sY_1}_{|C_1^\circ}\cong \sY^\circ\times_{C^\circ}C_1^\circ$
and ${\gamma_1}_{|g_1^{-1}(C_1^\circ)}$ is induced by $\gamma^\circ$.
\item Any fiber of $g_1$ has klt singularities.
\end{enumerate}
\end{lemma}

\begin{proof}
Let $\sY$ be a reduced and irreducible variety, and let 
$g \colon \sY \to C$ be a projective morphism onto $C$ such that $\sY_{|g^{-1}(C^\circ)}\cong \sY^\circ$ as varieties over $C^\circ$. We may also assume that $\gamma^\circ$ extends to a generically finite morphism
$\gamma\colon \sY \to \sX$.
By a theorem of Kempf, Knudsen, Mumford, and Saint-Donat (\cite{kkmsd}), there exist a smooth curve $C_1$,  
a finite morphism $\pi\colon C_1 \to C$, and a birational projective morphism $\wh \sY_1 \to \sY\times_C C_1$ from a smooth variety $\wh\sY_1$ such that the induced fibration $\wh \sY_1 \to C_1$ is semi-stable.
We may also assume without loss of generality that $\wh \sY_1 \to \sY\times_C C_1$ induces an isomorphism 
over $C_1^\circ:=\pi^{-1}(C^\circ)$.
Write $\sX_1:=\sX \times_C C_1$, and
consider the Stein factorization $\wh\sY_1 \to \sY_1 \to \sX_1$ of $\wh \sY_1 \to \sX_1$. We have the following commutative diagram:

\centerline{
\xymatrix{
\wh \sY_1\ar@/_1pc/[ddr]_{\textup{semi-stable}}\ar[dr]^{\textup{birational}}\ar[rr] & & \sY\times_C C_1 \ar[rd]\ar[rr] & & \sY \ar[rd]^{\gamma}\ar@/^3pc/[ddr]^{g} & \\
& \sY_1\ar[d]_{g_1}\ar[rr]^{\gamma_1,\,\textup{finite}}& & \sX_1=\sX\times_C C_1\ar[d]^{f_1}\ar[rr] & & \sX \ar[d]^{f}\\
& C_1 \ar@{=}[rr]& & C_1 \ar[rr]_{\pi} & & C. \\
}
}
\noindent Note that ${\gamma_1}_{|g_1^{-1}(C_1^\circ)}$ is a quasi-\'etale cover. Since $\wh \sY_1 \to C_1$ has reduced fibers, it follows that $\gamma_1$ is \'etale in codimension one. 

We still have to show that fibers of $g_1$ are klt.
Note first that $\sX_1$ is normal by Fact \ref{remark:normality}. This implies that 
$K_{\sX_1}$ is well-defined and $\mathbb{Q}$-Cartier since $K_{\sX_1/C_1}$ is then the pull-back of
$K_{\sX/C}$ under the projection morphism $\sX_1=\sX \times_C C_1 \to \sX$.
The claim now follows from Lemma \ref{lemma:quasi_etale_cover_singularities_fibers}, completing 
the proof of the lemma.
\end{proof}

We are now in position to prove Proposition \ref{prop:decomposition}.

\begin{proof}[Proof of Proposition \ref{prop:decomposition}]
We maintain notation and assumptions of Proposition \ref{prop:decomposition}.
Note that the statement (2) in Proposition~\ref{prop:decomposition} is an immediate consequence of (1) together with Proposition~\ref{prop:holonomy}.

By Proposition \ref{proposition:smoothing_disk_versus_alg_curve}, replacing $X$ with a quasi-\'etale cover, if necessary, we may assume that there exist
a normal variety $\sX$ and
a flat projective morphism with connected fibers $f \colon \sX \to C$
onto a smooth connected algebraic curve $C$ such that 
$X \cong f^{-1}(t_0)$ for some point $t_0$ on $C$, $f^{-1}(t)$ is smooth for $t \neq t_0$, and such that
$K_{\sX/C}\sim_\mathbb{Z} 0$.
Set $C^\circ:=C \setminus \{t_0\}$ and $\sX^\circ:=\sX\setminus \sX_{t_0}$.

Let $K$ be an algebraic closure of the function field of $C$. Applying the Beauville-Bogomolov decomposition theorem to $\sX_K$ (see \cite{beauville83}), we see that, 
replacing $C$ by a finite cover of some open neighborhood of $t_0$, if necessary, there exists 
an abelian scheme $\sB^\circ/C^\circ$, as well as 
finitely many families $(\sY_{i}^\circ/C^\circ)_{1 \le i \le s}$ of projective manifolds, and a finite \'etale cover 
$$\gamma^\circ\colon \sB^\circ\times_{C^\circ}\sY_{1}^\circ\times_{C^\circ}\cdots 
\times_{C^\circ}\sY_{s}^\circ \to  \sX^\circ$$
such that the $(\sY_{i}^\circ)_K$ are irreducible and simply-connected Calabi-Yau, or
symplectic manifolds. In particular, we have 
$\pi_1^{\textup{\'et}}\big((\sY_{i}^\circ)_K\big)=\{1\}$.
Applying \cite[Expos\'e X, Th\'eor\`eme 3.8]{sga1}, we conclude that $\pi_1^{\textup{\'et}}\big((\sY_{i}^\circ)_t\big)=\{1\}$ for any point $t$ on $C^\circ$.
On the other hand, we have $h^1\Big((\sY_{i}^\circ)_t,\sO_{(\sY_{i}^\circ)_t}\Big)=0$ by Lemma 
\ref{lemma:continuity} and $K_{(\sY_{i}^\circ)_t}\sim_\mathbb{Z}0$ by the adjunction formula.
It follows from \cite{beauville83} that $(\sY_{i}^\circ)_t$ has finite 
fundamental group for $t \neq t_0$. Hence $(\sY_{i}^\circ)_t$ is simply-connected.
Using \cite{beauville83} and Lemma \ref{lemma:continuity} again, we obtain that for any 
$t\in C^\circ$, the $(\sY_{i}^\circ)_t$ are irreducible and simply-connected Calabi-Yau, or
symplectic manifolds.

Write $\sY^\circ:=\sY_{1}^\circ\times_{C^\circ}\cdots 
\times_{C^\circ}\sY_{s}^\circ$, and $\sZ^\circ:=\sB^\circ\times_{C^\circ}\sY^\circ$.
By Lemma \ref{lemma:limit_etale_cover}, replacing $C$ by a further cover, if necessary, we may also assume that there exists a normal variety $\sZ$ as well as a flat morphism $g \colon \sZ \to C$ whose fibers have klt singularities, and a quasi-\'etale cover $\gamma\colon \sZ \to \sX$ such that $g^{-1}(C^\circ)=\sZ^\circ$ and $\gamma^\circ=\gamma_{|g^{-1}(C^\circ)}$. We may also assume that $h^\circ\colon \sY^\circ \to C^\circ$ has a section. Together with
the neutral section $0^\circ \colon C^\circ \to \sB^\circ$, we obtain a section of $g_{|g^{-1}(C^\circ)}$, and hence a section of $g$. Moreover, the projection morphism $\sZ^\circ \to \sB^\circ$ identifies with the natural morphism $\sZ^\circ \to \big(\Pic^\circ(\sZ^\circ/C^\circ)\big)^\vee$ induced by the universal line bundle.
By Proposition \ref{proposition:Abelian_factor}, there exists an abelian scheme $\sA \to C$, 
as well as a normal variety $\sY \subset \sZ$, and a finite \'etale cover
$\sY \times_C \sA \to  \sZ$ such that the natural map $h\colon \sY \to C$ is
a flat projective morphism  with connected normal fibers, and such that $h^{-1}(C^\circ)=0^\circ\times_{C^\circ}\sY^\circ\cong\sY^\circ$. The situation is summarized in the following diagram: 
  $$
  \xymatrix{ 
    \sA \times_C \sY \ar[r]_{\text{étale}}  & \sZ  \ar[r]^{\gamma}_{\text{quasi-étale}} & \sX \ar[r] & C \\
     & \sZ^\circ \cong \sB^\circ \times_{C^\circ} \sY^\circ \ar@{}[u]|{\bigcup} \ar[r] &\sX^\circ \ar@{}[u]|{\bigcup}\ar[r] &   C^\circ.  \ar@{}[u]|{\bigcup}
  }
  $$
 
Let $\sL$ be a relatively ample line bundle on $\sY$. Denote by $s_i^\circ \colon C^\circ \to \sY_i^\circ$ the induced section of $\sY_i^\circ \to C^\circ$, and consider the embedding $\sY_i^\circ \subset \sY_{1}^\circ\times_{C^\circ}\cdots 
\times_{C^\circ}\sY_{s}^\circ\cong \sY^\circ$ induced by the sections $s_j$ for $j\neq i$.
Denote by 
$p_i^\circ \colon \sY^\circ \to \sY_i^\circ$ the projection morphism. Then, we have
$$\sL_{|\sY^\circ} \cong (p_1^\circ)^*(\sL_{|\sY_1^\circ})\otimes \cdots \otimes 
(p_s^\circ)^*(\sL_{|\sY_s^\circ})$$
since $h^1\Big((\sY_{i}^\circ)_t,\sO_{(\sY_{i}^\circ)_t}\Big)=0$ for any point $t$ on $C^\circ$.

Suppose now that $\sL$ is very ample over $C$, and set 
$N_i:=h^0\big((\sY_i^\circ)_t,\sO_{(\sY_i^\circ)_t}\big)-1$. Then $\sL_{|\sY_i^\circ}$ is very ample over $C^\circ$ and induces an embedding $\sY_i^\circ\subset \bP^{N_i}\times C^\circ$. It follows that
$\sY_{1}^\circ\times_{C^\circ}\cdots 
\times_{C^\circ}\sY_{s}^\circ$ embeds into $\bP^{N_1}\times\cdots\times\bP^{N_s}\times C^\circ$. Set $N:=h^0(\sY^\circ,\sO_{\sY^\circ})-1=(N_1+1)\cdots(N_s+1)-1$. Then $\sL$ embeds $\sY$ into $\bP^N\times C$ and the image of $\sY^\circ$ agree with the image of 
$\sY_{1}^\circ\times_{C^\circ}\cdots 
\times_{C^\circ}\sY_{s}^\circ$ under the Segre embedding 
$\bP^{N_1}\times\cdots\times\bP^{N_s}\times C^\circ \subset \bP^N\times C^\circ$.
Let $\sY_i \subset \bP^{N_i}\times C$ be the closure of $\sY_i^\circ$. Note that
$\sY_i \to C$ is flat. Since the scheme $\textup{Hilb}(\bP^N)$ is separated, we conclude that
$(\sY_{1})_{t_0}\times\cdots \times(\sY_{s})_{t_0} = {\sY}_{t_0}\subset \bP^N$. This easily implies that
$(\sY_{i})_{t_0}$ is a normal projective variety with klt singularities and trivial canonical divisor, completing the proof of Proposition \ref{prop:decomposition}. 
\end{proof}

\begin{proof}[Proof of Theorem B]
By Proposition \ref{proposition:smoothing_disk_versus_alg_curve}, we may assume without loss of generality that $X$ admits a 
projective $\mathbb{Q}$-Gorenstein 
smoothing. Theorem B is now an easy consequence of Propositions~\ref{prop:decomposition} and \ref{prop:holonomy}. Indeed, the only thing to check is the assertion concerning the algebra of reflexive forms. Proposition~\ref{prop:extension_differential_forms} settles the Calabi-Yau case immediately. In the symplectic case, item (2) in Proposition~\ref{prop:holonomy} together with the Bochner principle (see Theorem~\ref{holprin}) yield a reflexive $2$-form $\sigma$ on $X$, symplectic on $X_{\textup{reg}}$,   
while Proposition~\ref{prop:extension_differential_forms} shows that $\sigma$ generates the algebra of reflexive forms. 
\end{proof}

\subsection{Irreducible Calabi-Yau and symplectic varieties with stable tangent sheaf}
\label{ssec:factors}

In this section, we try to analyze a bit further the factors appearing in the decomposition of $X'$ in Theorem~B. These varieties have stable tangent sheaf, but that sheaf may not be strongly stable. Such varieties are conjecturally covered by either an abelian variety or a product of copies of a single irreducible, Calabi-Yau or symplectic variety, see Conjecture~\ref{conj}. Assuming that a weak singular analogue of the Beauville-Bogomolov decomposition theorem holds, we prove this conjecture.

\begin{prop}\label{prop:conj}
Suppose that any projective variety $X$ with klt singularities and numerically trivial canonical class admits a quasi-étale cover $Y\to X$ that splits as a product of an abelian variety and varieties with strongly stable tangent sheaves.
Then Conjecture~\ref{conj} holds. 
\end{prop}

\begin{proof}
We maintain notation and assumptions of Conjecture~\ref{conj}. We know that there exists a quasi-\'etale cover $Y \to X$  
of $X$ such that $Y$ decomposes as a product $Y \cong Y_0 \times Y_1\times\cdots\times Y_m$ of an abelian variety $Y_0$ and 
varieties $Y_i$ for $1 \le i \le m$ with strongly stable tangent sheaves. We may assume without loss of generality that 
$\dim Y_i \ge 2$ for $1 \le i \le m$.
Let $Z \to Y$ be a quasi-\'etale cover such that the induced quasi-\'etale cover $Z \to X$ is Galois, with Galois group $G$. 

The decomposition $Y \cong Y_0 \times Y_1\times\cdots\times Y_m$ induces a decomposition
$$T_Z\cong \sE_0\oplus \sE_1\oplus\cdots\oplus\sE_m$$ of $T_Z$, where the $\sE_i$ for $1\le i \le m$ are
strongly stable sheaves.
Observe that the foliation $\sE_i$ is induced by the Stein factorization of the projection $$\pi_i\colon Z \to B_i:=Y_0\times \cdots Y_{i-1}\times Y_{i+1}\times\cdots\times Y_m.$$  Note also that $$\sE_0\oplus \cdots \oplus\sE_{i-1}\oplus\sE_{i+1}\oplus\cdots\oplus\sE_m$$
induces a flat connection on $\pi_i$. A classical result of complex analysis then implies that $\pi_i$ is a locally trivial 
fibration for the Euclidean topology over the smooth locus of $B_i$. Denote by $Z_i$ a connected component of general fiber of $\pi_i$. It comes with a quasi-\'etale cover
$Z_i \to Y_i$.

Suppose first that $\dim Y_0 \ge 1$. Since $\sE_i$ is strongly stable for $1 \le i \le m$, we have
$h^0(Z,\sE_i)=0$ for each $1 \le i \le m$. This immediately implies that $\sE_0$ is stable under $G$, and thus $T_Z=\sE_0$ since $T_X$ is stable. This shows that $Y$ is an abelian variety. 

Suppose from now on that $\dim Y_0 = 0$. We claim that $\sE_i \not\cong \sE_j$ if $i \neq j$. Indeed, we have
${\sE_j}_{|Z_i}\cong \sO_{Z_i}^{\oplus \,\textup{rank}\sE_j}$ while 
$\big({\sE_i}_{|Z_i}\big)^{**}\cong T_{Z_i}$ and $h^0(Z_i,T_{Z_i})=0$ since $T_{Z_i}$ is strongly stable.
Therefore, the group $G$ acts on the set $1\le i \le m$ of stable summands of $T_Z$, and since $T_X$ is stable, this action is transitive. Thus, for any $i\in I$, there exists $g_i\in G$ such that $\sE_i\cong g_i^*\sE_1 \subset g_i^*T_Z\cong T_Z$. 
This in turn implies that $Z_i \cong Z_1$, completing the proof of the proposition.
\end{proof}

To finish this section, let us rephrase some results obtained in Proposition~\ref{prop:holonomy} with the notations of Theorem~B. We start by fixing on $X$ a singular Ricci-flat Kähler metric, provided by \cite{EGZ}. It can be showed (see \cite[Proof of Proposition~7.6]{GGK}) that the induced metric on $A\times  X'$ is actually a product metric, and that the metric induced on $ X'$ is also a product metric compatible with the decomposition of $ X'$. Up to passing to a further cover and inflating the abelian part, one can assume that the factors of $ X'$ have non-trivial restricted holonomy. Finally, one can pass to an holonomy cover by \cite[Theorem B]{GGK} to ensure that the holonomy is connected. Piecing everything together, Proposition~\ref{prop:holonomy} shows that a quasi-étale cover of $X$ splits as $$A \times \prod_{i\in I} Y_i \times \prod_{j\in J} Z_j$$ where $A$ is an abelian variety, and such that the following holds. Let $y_i \in Y_i$ and $z_j\in Z_j$ be smooth points.

\begin{enumerate}
\item For every $i\in I$, there exist $n_i,r_i\in \mathbb N$ with $n_ir_i = \dim Y_i$ such that the holonomy acts on $T_{y_i}Y_i$ by the standard product representation $\SU(n_i)^{\times r_i} \circlearrowleft \mathbb C^{\dim Y_i}$.  
\item For every $j\in J$, there exist $n_j,r_j\in \mathbb N$ with $2n_jr_j=\dim Z_j$ such that the holonomy acts on $T_{z_j}Z_j$ by the standard product representation $\Sp(n_j)^{\times r_j}\circlearrowleft \mathbb C^{\dim Z_j}$.
\end{enumerate}

\section{Proof of Theorem A} 
\label{sec5}

In this section we prove Theorem A. 

\medskip

Let $X$ be a (proper) variety, let $m$ be a positive integer,
and write $A:=\mathbb{C}[[x_1,\ldots,x_m]]$ and $K:=\mathbb{C}((x_1,\ldots,x_m))$. Let also $\wb K$ be an algebraic closure of $K$. Recall that \textit{a deformation of $X$ over $A$} is a flat morphism of schemes 
$f \colon \mathcal{X} \to \textup{Spec}\, A$ such that $\mathcal{X}\otimes(A/\mathfrak{m})\cong X$, where $\mathfrak{m}$ denotes the maximal ideal of $A$. We say that $f$ is a \textit{proper deformation of $X$ over $A$} if $f$ is a proper morphism. 

If $f$ is a smooth proper morphism with geometrically connected fibers, then there is an isomorphism of fundamental groups
$\pi_1^{\textup{\'et}}\big(\mathcal{X}_{\wb K}\big) \cong \pi_1^{\textup{\'et}}\big(X\big)$ (see \cite[Expos\'e X, Th\'eor\`eme 3.8]{sga1}). However, it is well-known that this statement becomes wrong if $f$ is not assumed to be proper.
The following will prove to be crucial.

\begin{thm}\label{thm:fundamental_group}
Let $X$ be a normal proper variety with klt singularities, and assume that $X$ is smooth in codimension two. Write $A:=\mathbb{C}[[x_1,\ldots,x_m]]$ and $K:=\mathbb{C}((x_1,\ldots,x_m))$, and let $\wb K$ be an algebraic closure of $K$.
Let $\mathcal{X}$ be a proper deformation of $X$ over $A$. 
Suppose that $\mathcal{X}_{\wb K}$ is smooth with $\pi_1^{\textup{\'et}}\big(\mathcal{X}_{\wb K}\big)=\{1\}$. Then $\pi_1^{\textup{\'et}}\big(X_{\textup{reg}}\big)=\{1\}$.
\end{thm}

The following example shows that Theorem \ref{thm:fundamental_group} is wrong if one relaxes the assumption on the codimension of the singular locus.

\begin{exmp}
Let $X \subset \mathbb{P}^3$ be a cone over a smooth plane cubic curve, and let $f\colon \sX \to \mathbb{P}^1$ be flat family of cubic surfaces in $\mathbb{P}^3$ such that $f^{-1}(0)=X$ and $f^{-1}(t)$ is smooth for a general point $t$ on $\mathbb{P}^1$. Then $\pi_1\big(\sX_{\wb{\mathbb{C}(t)}}\big)=\{1\}$ and 
$\pi_1\big(X_{\textup{reg}})\cong \mathbb{Z}\oplus\mathbb{Z}$.

\end{exmp}

The same arguments used in the proof of Lemmas \cite[I.9.1 and I.9.2]{artin_def_sing}
show that the following holds.

\begin{lemma}\label{lemma:extension}
Let $X$ be a normal variety of dimension at least two, let $X^\circ \subset X_{\textup{reg}}$ be an open subset, and let $A$ be a local artinian $\mathbb{C}$-algebra. Let $i\colon X^\circ \into X$ be the inclusion map.
\begin{enumerate}
\item Let $X_A$ be a deformation of $X$ over $A$, and suppose that $X \setminus X^\circ$ has codimension at least two. Then the restriction map induces an isomorphism $\sO_{X_A} \cong i_*\big(i^{-1}\sO_{X_A}\big)$. 
Let $X_A^1$ and $X_A^2$ be deformations of $X$ over $A$, and let 
$(X_A^1)^\circ \subset X_A^1$ and $(X_A^2)^\circ \subset X_A^2$ be the open subschemes with underlying topological space $X^\circ$.
Then any isomorphism $(X_A^1)^\circ\cong (X_A^2)^\circ$ of $A$-schemes uniquely extends to an isomorphism 
$X_A^1\cong X_A^2$ of $A$-schemes. 
\item Suppose that $X$ is affine, that $X \setminus X^\circ$ has codimension at least three, and that $\textup{depth}\,\sO_{X,x} \ge 3$ for any point $x \in X\setminus X_{\textup{reg}}$. Let $X^\circ_A$ be a deformation of $X^\circ$ over $A$, and set 
$X_A:=\textup{Spec}\, H^0(X^\circ_A,\sO_{X^\circ_A})$. Then $X_A$ is a flat deformation of $X$ over $A$ extending $X^\circ_A$.
\end{enumerate}
\end{lemma}

\begin{prop}\label{proposition:extension_deformation_q_etale_covers}
Let $\gamma\colon Y \to X$ be a quasi-\'etale cover of normal proper varieties. Suppose that
$X$ is smooth in codimension two, 
and that $\textup{depth}\,\sO_{Y,y} \ge 3$ for any point 
$y \in Y\setminus Y_{\textup{reg}}$. Let $(A,\mathfrak{m})$ be a complete local noetherian $\mathbb{C}$-algebra with residue field $\mathbb{C}$, and let $\mathcal{X}$ be a proper deformation of $X$ over $A$. 
Then there exists a proper deformation $\mathcal{Y}$ of $Y$ over $A$, and a finite morphism $\Gamma\colon \mathcal{Y} \to \mathcal{X}$ extending $\gamma$. If moreover $A$ is regular, then $\Gamma$ is a quasi-\'etale cover. 
\end{prop}

\begin{proof}
For any non-negative integer $m$, write $A_m:=A/\mathfrak{m}^{m+1}$, $S:=\textup{Spec}\, A$,
$S_m:=\textup{Spec}\, A_m$, and $X_m:=\mathcal{X} \times_S S_m$. We will also denote by $\wh{\mathcal{X}}$ the formal completion of $\mathcal{X}$ along $X$. Note that $\wh{\mathcal{X}}$ is the colimit of the $X_m$. Moreover, it comes with a proper morphism onto $\wh S := \textup{Spf}\,A$.

Let $X^\circ$ be the smooth locus of $X$, and set $Y^\circ:=\gamma^{-1}(X^\circ)$. By the Nagata-Zariski purity theorem, $\gamma$ branches only over the singular set of $X$, and hence $\gamma_{|Y^\circ}\colon Y^\circ \to X^\circ$ is an \'etale cover. Note that $Y \setminus Y^\circ$ has codimension at least three and that $Y^\circ \subset Y_{\textup{reg}}$.

Let $i_m\colon X_m^\circ \subset X_m$ be the open subscheme with underlying topological space $X^\circ$. By \cite[Th\'eor\`eme 18.1.2]{ega32}, there exists a finite \'etale cover $\gamma_m^\circ\colon Y_m^\circ \to X_m^\circ$  
such that $Y_m^\circ \cong Y^\circ_{m+1}\times_{S_{m+1}} S_m$ for any integer $m \ge 0$. 
Set $Y_m:=\textup{Spec}_{X_m}(i_m\circ \gamma_m^\circ)_*\sO_{Y_m^\circ}$, and denote 
by $\gamma_m\colon Y_m \to X_m$ the natural morphism. Note that $Y_m^\circ$ can be identified with the open subscheme of $Y_m$ with underlying topological space $Y^\circ$.
Applying Lemma \ref{lemma:extension} above, we see that $Y_m$ is flat over $S_m$, and that
$\gamma_{m+1}\times_{S_{m+1}S_m}=\gamma_m$.
Moreover, by \cite[Expos\'e VIII, Corollaire 2.3]{sga2}, $\gamma_m$ is finite. 
Let $\wh{\mathcal{Y}}$ denote the colimit of the $Y_m$; $\wh{\mathcal{Y}}$ is a noetherian formal scheme, proper and flat 
over $\wh S$. 
It comes with a finite morphism $\wh \Gamma \colon \wh{\mathcal{Y}} \to \wh{\mathcal{X}}$.
Now, by \cite[Proposition 5.4.4]{ega11}, 
$\wh{\mathcal{Y}}$ is algebraizable. More precisely, there exists a scheme $\mathcal{Y}$ proper over $S$ such that 
$\wh{\mathcal{Y}}$ identifies with the formal completion of $\mathcal{Y}$ along $Y$. 
Moreover, the morphism $\wh\Gamma \colon \wh{\mathcal{Y}} \to \wh{\mathcal{X}}$ of formal schemes is induced by a finite morphism of schemes $\Gamma\colon\mathcal{Y} \to \mathcal{X}$ (see proof of \cite[Proposition 5.4.4]{ega11}). Note that $\mathcal{Y}$ is flat over $S$ by \cite[Chapitre III, \S 5, Th\'eor\`eme 2 et Proposition 2]{bourbaki3-4}.

Suppose from now on that $A$ is regular.
Let $x \in \mathcal{X}$ be a codimension one point, not contained in the special fiber $\mathcal{X}_\mathfrak{m}$, and suppose that $\Gamma$ is ramified at $x$.
Then $\gamma$ must be ramified along $\wb{\{x\}}\cap \mathcal{X}_{\mathfrak{m}}$.
On the other hand, $\wb{\{x\}}\cap \mathcal{X}_{\mathfrak{m}}$ has codimension one in $\mathcal{X}_{\mathfrak{m}}$ since $A$ is regular. This yields a contradiction, and shows that $\Gamma$ is a quasi-\'etale cover, completing the proof of the proposition. 
\end{proof}

Before proving Theorem \ref{thm:fundamental_group} below, we note the following consequence of Proposition \ref{proposition:extension_deformation_q_etale_covers}.

\begin{cor}\label{cor:above}
Let $X$ be a normal proper variety with klt singularities, and assume that $X$ is smooth in codimension two. Write $A:=\mathbb{C}[[x_1,\ldots,x_m]]$, and let $\mathcal{X}$ be a proper deformation of $X$ over $A$. 
Then the natural map 
$\pi_1^{\textup{\'et}}\big(X_{\textup{reg}}\big) \to \pi_1^{\textup{\'et}}\big(\mathcal{X}_{\textup{reg}}\big)$
is injective.
\end{cor}

\begin{proof}
Given a finite \'etale cover $\gamma^\circ\colon Y^\circ \to X_{\textup{reg}}$ with $Y^\circ$ connected, we need to show that there exists a finite \'etale cover $\Gamma^\circ\colon\mathcal{Y}^\circ \to \mathcal{X}_{\textup{reg}}$ such that
$Y^\circ$ is a connected component of ${\mathcal{Y}^\circ}_{|(\Gamma^\circ)^{-1}(X_{\textup{reg}})}$ and such that $\gamma^\circ$ is induced by $\Gamma^\circ$.

Let $\gamma\colon Y \to X$ be the normalization of $X$ in the function field of $Y^\circ$. Note that $\gamma$ is a quasi-\'etale cover.
Applying \cite[Proposition 3.16]{kollar97} to $\gamma$, we see that
$Y$ has klt singularities. It follows that $Y$ is Cohen-Macaulay by \cite[Theorem 1.3.6]{kmm}.
By the Nagata-Zariski purity theorem, $\gamma$ branches only over the singular set of $X$, and hence $Y$ is smooth in codimension two. Combining the previous two assertions, one sees that the assumption from Proposition \ref{proposition:extension_deformation_q_etale_covers} about the depth of points in $Y\smallsetminus Y_{\reg}$ is satisfied. Applying the aforementioned proposition then proves the corollary.
\end{proof}

\begin{proof}[Proof of Theorem \ref{thm:fundamental_group}]
By the semicontinuity theorem, we have 
$h^0(\mathcal{X}_{\wb K},\sO_{\mathcal{X}_{\wb K}})=h^0(X,\sO_{X})=1$, and hence $\mathcal{X}_{\wb K}$ is connected.

Let $\gamma^\circ\colon Y^\circ \to X_{\textup{reg}}$ be a finite \'etale cover with $Y^\circ$ connected, and
let $\gamma\colon Y \to X$ be the normalization of $X$ in the function field of $Y^\circ$. Note that $\gamma$ is a quasi-\'etale cover. The same argument used in the proof of Corollary \ref{cor:above} shows that 
there exists a deformation $\mathcal{Y}$ of $Y$ over $A$ and a quasi-\'etale cover $\Gamma\colon \mathcal{Y} \to \mathcal{X}$ over $A$ extending $\gamma$.
By the semicontinuity theorem again, we see that $\mathcal{Y}_{\wb K}$ is connected. Moreover, by the Nagata-Zariski purity theorem, the finite morphism 
$\gamma_{\wb K}\colon \mathcal{Y}_{\wb K} \to \mathcal{X}_{\wb K}$ is \'etale, and hence an isomorphism since 
$\pi_1^{\textup{\'et}}\big(\mathcal{X}_{\wb K}\big)=\{1\}$. This implies that $\gamma$ is an isomorphism as well, completing the proof of the theorem.
\end{proof}

We will also need the following observation. The proof follows the line of argument given in \cite[Corollary 12.1.9]{kollar_mori_flips}.

\begin{lemma}\label{lemma:q_gorenstein}
Let $X$ be a normal projective variety with klt singularities. If $X$ is smooth in codimension two, then 
any projective smoothing over an algebraic curve is a $\mathbb{Q}$-Gorenstein smoothing. 
\end{lemma}

\begin{proof}
Let $\sX$ be a normal variety, and let $f \colon \sX \to C$ be a flat projective morphism with connected fibers
onto a smooth connected curve $C$ such that $X \cong f^{-1}(t_0)$ for some point $t_0$ on $C$ and such that $f^{-1}(t)$ is smooth for $t \neq t_0$.

By \cite[Theorem 1.3.6]{kmm}, we know that
$X$ is Cohen-Macaulay. Let $m$ be a positive integer such that $mK_X$ is a Cartier divisor. 
Set $U:=\sX \setminus (X \setminus X_{\textup{reg}}) \subset \sX_{\textup{reg}}$, and
denote by $j\colon U \into \sX$ and $i\colon X_{\textup{reg}}=\sX_{t_0}\cap U \into X$
the natural inclusions. By \cite[Proposition 1.6]{hartshorne80}, 
we have $j_*\big({\sO_\sX(mK_{\sX/C})}_{|U}\big) \cong \sO_\sX(mK_{\sX/C})$
and 
$i_*\big(\sO_\sX(mK_{\sX/C})_{|X_{\textup{reg}}}\big)\cong i_*\big(\sO_X(mK_X)_{|X_{\textup{reg}}}\big)\cong \sO_X(mK_X)$. Applying
\cite[Lemma 12.1.8]{kollar_mori_flips}, we see that
${\sO_\sX(mK_{\sX/C})}_{|X} \cong \sO_X(mK_X)$. This implies
that ${\sO_\sX(mK_{\sX/C})}$ is a Cartier divisor, proving the lemma.
\end{proof}

We are now in position to prove our main result.

\begin{proof}[Proof of Theorem~A]
We maintain notation and assumptions of Theorem~A.

Applying Lemma \ref{lemma:smoothing_disk_versus_alg_curve2} and Lemma \ref{lemma:q_gorenstein}, we see that we may assume 
without loss of generality that $X$ admits a 
projective $\mathbb{Q}$-Gorenstein 
smoothing $f \colon \sX \to C$ over an algebraic curve $C$. Let $t_0$ be a point on $C$ such that $X \cong f^{-1}(t_0)$.
By Proposition \ref{prop:decomposition} and Lemma \ref{lemma:q_gorenstein} again, we may also assume that $\sX_t$ is an irreducible and simply-connected Calabi-Yau, or
symplectic manifold for any point $t \neq t_0$, and that $T_X$ is slope-stable with respect to any ample polarization on $X$.

Denote by $A$ the completion of the local ring $\sO_{C,t_0}$. Note that $A\cong \mathbb{C}[[t]]$ by a theorem of Cohen. Let $K$ denotes the field of fractions of $A$, and let $\wb K$ be an algebraic closure of $K$.
Write $S:=\textup{Spec}\, A$, and
$\mathcal{X}:= \sX\times_C S$. 

Applying \cite[Expos\'e X, Th\'eor\`eme 3.8]{sga1}, we see that 
$\pi_1^{\textup{\'et}}\big(\mathcal{X}_{\wb K}\big)=\{1\}$. It follows 
from Theorem \ref{thm:fundamental_group}
that
$\pi_1^{\textup{\'et}}\big(X_{\textup{reg}}\big)=\{1\}$, and hence
the tangent sheaf $T_X$ is strongly stable. The theorem now follows from Theorem~\ref{thm:ggk}.
\end{proof}

\begin{rem}
There is a alternative proof of Theorem~A in the case where $\X_t$ is irreducible symplectic, still assuming the $X:=\X_0$ is smooth in codimension two. Indeed, in that case a theorem of Namikawa \cite[Corollary 2]{Namikawa06} asserts that 
if $\pi\colon Y\to X$ is a $\mathbb Q$-factorial terminalization of $X$, then $Y$ is smooth and $\pi$ is a symplectic resolution. Note that $\mathbb Q$-factorial terminalizations always exist for varieties with canonical singularities by \cite[Corollary 1.4.3]{bchm}. By a result of Kaledin (see \cite[Proposition 1.2]{Kaledin}), a symplectic resolution is a semi-small morphism; in particular, $2\,\codim\,\pi^{-1}(X_{\rm sing}) \ge \codim\, X_{\rm sing}$, and therefore $\codim\, \pi^{-1}(X_{\rm sing})\ge 2$. Set $Y^{\circ}:=Y\smallsetminus \pi^{-1}(X_{\rm sing})$, and note that $Y^{\circ}\cong X_{\reg}$ and $\codim\, (Y\smallsetminus Y^{\circ}) \ge 2$. In particular, we have $\pi_1(Y^{\circ})\cong \pi_1(Y)$. 
By \cite[Theorem 1.1]{Takayama2003}, we also have $\pi_1(Y)\cong \pi_1(X)$, and from \cite[Lemma 5.2.2]{Kollar93}), we see that
$\pi_1(X)=0$. Eventually, $X_{\reg}$ is simply-connected, and as $\sT_X$ is stable by Proposition~\ref{prop:holonomy}, it is automatically strongly stable. 
\end{rem}

\section{Examples}
\label{sec6}
In this section, we first give examples of smoothable (irreducible) Calabi-Yau and symplectic varieties. We also
collect examples which illustrate to what extent our results are sharp. 
We maintain notation of Section \ref{sec3}.

\subsection{Examples of smoothable Calabi-Yau varieties}
\begin{exmp}[Nodal hypersurfaces]
Let $X$ be a nodal degree $n+2$ hypersurface in $\mathbb P^{n+1}$ with $n\ge 3$. This means that the singularities of $X$ are isolated and locally analytically isomorphic to the germ at $0$ of the quadric 
$$Q:=\left\{z\in \mathbb C^{n+1} \, | \, \sum_{i=1}^{n+1}z_i^2=0\right\}.$$ Then $X$ has canonical Gorenstein singularities, and it is smooth in codimension $2$. Moreover, $X$ has trivial canonical bundle and is smoothable by irreducible Calabi-Yau manifolds. By Theorem~\ref{thm:fundamental_group}, one has $\pi_1^{\textup{\'et}}(X_{\reg})=\{1\}$. Then, Proposition~\ref{prop:extension_differential_forms} shows that $X$ is an irreducible Calabi-Yau variety.

This applies in particular to $$X:=\left\{[x_0:\cdots : x_4] \in \mathbb P^4 \,| \, x_0f+x_1g=0\right\}$$ where $f$ and $g$ are general homogeneous polynomials in $x_0, \ldots, x_4$ of degree $4$. Indeed, it is clear that $X$ has $16$ (isolated) nodal singularities. In that example, there is an alternative proof that bypasses Theorem~\ref{thm:fundamental_group}. Indeed, blowing up the plane $S:=(x_0=x_1=0)\subset \mathbb P^4$ yields a small resolution $\pi\colon \wt X\to X$ so in particular, we must have $\pi_1(X_{\reg})\cong \pi_1\big(\wt X\smallsetminus \pi^{-1}(S)\big)\cong \pi_1(\wt X)$ since $\codim\,\pi^{-1}(S) \ge2$. Applying 
successively \cite[Theorem 1.1]{Takayama2003} and the Lefschetz's hyperplane theorem for fundamental groups, we see that
$\pi_1(\wt X) \cong \pi_1(X)=\{1\}$, and hence $\pi_1(X_{\reg})=\{1\}$ as claimed.  
\end{exmp}

\begin{exmp}[Calabi-Yau threefolds with non-nodal singularities]
\label{ex1}
Let $S_1\to \mathbb P^1$ and $S_2\to \mathbb P^1$ be rational elliptic surfaces with sections, and let 
$X:=S_1 \times_{\mathbb P^1} S_2$ (see \cite{Schoen88}). 
In \cite[Example 5.9]{Namikawa94}, Namikawa gives conditions on the singular fibers of each fibration under which $X$ is a smoothable $\mathbb Q$-factorial 3-fold with trivial canonical class with one isolated terminal singularity locally analytically isomorphic to the germ at $0$ of the hypersurface $$\{z\in \mathbb C^4 \, | \, z_1^2+z_2^2+z_3(z_3+z_4)(z_3-z_4)=0\}.$$ As in the example above, one concludes that $X$ is an irreducible Calabi-Yau variety.
\end{exmp}

\begin{exmp}[Calabi-Yau threefolds with $\mathbb Q$-factorial isolated rational hypersurface singularities]
\label{ex2}
If $X$ is a projective variety of dimension three with canonical singularities and trivial canonical bundle, then Namikawa and Steenbrink proved in \cite[Theorem 1.3]{NS95} that $X$ admits a flat deformation to a smooth Calabi-Yau threefold provided that $X$ has only $\mathbb Q$-factorial, isolated, rational hypersurface singularities. 
If $X$ is singular, then it is an irreducible Calabi-Yau variety. The argument is as follows. 

Suppose from now on that $X$ is singular, and let $Y \to X$ be any quasi-\'etale cover. We claim that $Y$ is also singular. Suppose otherwise, and let $Z \to Y$ be a quasi-\'etale cover such that the induced finite morphism $Z \to X$ is Galois. By the Nagata-Zariski purity theorem, $Z\to Y$ is \'etale, and thus $Z$ is smooth as well. This in turn implies that $X$ has quotient singularities. On the other hand, any isolated quotient singularity in dimension at least three is rigid by 
\cite[Theorem 2]{Sch71} (see also Remark~\ref{rem:schl}), yielding a contradiction. 
By the Nagata-Zariski purity theorem again, we conclude that $Y$ has isolated singularities.

Applying Proposition~\ref{prop:decomposition}, we see that there exist a quasi-étale cover $Y \to X$ and a smoothing of $Y$ into 
irreducible and simply-connected Calabi-Yau or symplectic manifolds. By Proposition~\ref{prop:holonomy} and Theorem~\ref{thm:fundamental_group}, $Y$ is an irreducible Calabi-Yau variety, and hence so is $X$.

If the global assumption on $\mathbb Q$-factoriality is dropped, \cite[Theorem 2.4]{NS95} shows that one can still improve the singularities of $X$ by deforming it to a variety with only nodal singularities.    
\end{exmp}

\begin{rem}
It should be noted that from our classification point of view, threefolds are completely understood as they are known to satisfy a singular analogue of the Beauville-Bogomolov decomposition theorem by \cite{bobo}. Therefore, Examples~\ref{ex1} and \ref{ex2} should be thought as illustrative rather than new. 
For instance, \cite{bobo} can be used along the same lines as in Example~\ref{ex2} above to prove that if $X$ is a smoothable projective variety of dimension three with canonical isolated singularities and trivial canonical class, then $X$ is an irreducible Calabi-Yau variety. 
\end{rem}

\subsection{Examples of smoothable symplectic varieties}
\label{ex:symp}
Let $X$ be a normal variety.
Recall that we say that $X$ is a \textit{symplectic variety} if $X$ has canonical singularities and 
there exists $\omega \in H^0\big(X_{\reg}, \Omega^2_{X_{\reg}}\big)$ everywhere non-degenerate. If $\pi\colon \wt X \to X$ is a resolution of singularities of $X$, then $\omega$ extends to a holomorphic $2$-form on $\wt X$ by 
\cite[Theorem 1.4]{GKKP}. Let now $\pi\colon \wt X \to X$ be a $\mathbb Q$-factorial terminalization of $X$. Recall that the existence of
$\pi$ is established in \cite[Corollary 1.4.3]{bchm}. 
Then \cite[Corollary~2]{Namikawa06} asserts that $X$ is smoothable if and only if $\wt X$ is smooth.
Note that $\omega$ automatically extends to a symplectic form on $\wt X_{\reg}$ since 
$\pi$ is crepant. In particular, if $X$ is smoothable, then $X$ admits a symplectic resolution. Conversely, if $X$ admits a symplectic resolution, then $X$ is smoothable by \cite[Theorem~2.2]{Namikawa01}.
More precisely, any smoothing $X_t$ of $X$ by symplectic manifolds is a flat deformation of $\wt X$.

Note that Namikawa's theorem \cite[Corollary~2]{Namikawa06} provides smoothings by Kähler manifolds which are not projective.

\subsection{Examples of smoothable symplectic varieties with stable but not strongly stable tangent sheaf}\label{sec:nonsstable}

Recall from Proposition \ref{prop:holonomy} and Theorem~\ref{thm:fundamental_group} that if a normal projective variety $X$ with klt singularities and $K_X \equiv 0$ whose tangent sheaf is not strongly stable admits a projective smoothing into simply-connected irreducible Calabi-Yau, or symplectic manifolds then $\codim (X \smallsetminus X_{\rm reg})=2$.
 
\begin{exmp}[Singular Kummer surface]
\label{ex:kummer}
In the following example, we consider a degeneration of $K3$ surfaces to a singular Kummer surface $X$. 
More precisely, let $X=A/{\langle \pm 1\rangle}$ where $A$ is a principally polarized abelian surface. Then the following holds.

\medskip

\begin{itemize}
\item The Kummer surface $X$ admits a $\mathbb{Q}$-Gorenstein
projective smoothing $f \colon \sX \to C$ by $K3$ surfaces.
\item The tangent sheaf $T_X$ is not strongly stable.
\item Let $\mathscr L$ be a relatively ample line bundle on $\X$, and denote by $g$ the Ricci-flat K\"ahler metric on $X_{\reg}$ given by \cite[Theorem 7.5]{EGZ} applied to $(\sX_0,\sL_{|\sX_0})$. Given $x \in X_{\reg}$,
we have
$$\Hol(X_{\reg},g)_x\cong \mathbb Z/2\mathbb Z \quad \mbox{and} \quad \Hol^{\circ}(X_{\reg},g)_x=\{1\}.$$
\item We have $\pi_1(X)=\{1\}$ and $\pi_1(X_{\reg})$ is an extension of $\mathbb Z^4$ by $  \mathbb Z/2\mathbb Z$.
\end{itemize}

\medskip

It is well-known that $X$ can be realized as a singular quartic surface in $\P^3$ (see \cite[Theoremm 4.8.1]{birkenhake_lange}). In particular, it can be seen as the fiber over some point $t_0\in C$ of a projective flat family $f\colon\X\to C$ of quartic surfaces
over a smooth algebraic curve $C$ such that $\X_t$ is smooth if $t\neq t_0$. The total space $\X$ 
is an hypersurface in $C \times \mathbb{P}^3$, and hence $\mathbb Q$-Gorenstein. For $t\neq t_0$, $\X_t$ is a smooth quartic surface, and thus it is a $K3$ surface.

The tangent sheaf $T_X$ is not strongly stable as $X$ admits a quasi-étale cover $\gamma\colon A\to A/{\langle \pm 1\rangle}$ of degree two and $\sT_{A}\cong \sO_A^{\oplus 2}$ is obviously not stable. 

By \cite[Theorem 7.5]{EGZ}, there exists a unique closed positive $(1,1)$-current $\om$ with bounded potentials on $X$ such that $[\om] = [c_1(\sL_{|X})] \in H^2(X,\mathbb{R})$ and such that the restriction of $\omega$ to $X_{\reg}$ is a smooth Kähler form with zero Ricci curvature. Note that $g$ is the Riemannian metric associate with ${\om}_{|X_{\reg}}$ on $X_{\reg}$.
We claim that $g$ is flat, or equivalently that $\Hol^{\circ}(\Xr,g)_x=\{1\}$. Indeed, it is known that $\om$ is of orbifold type, that is, $\gamma^*\om$ defines a smooth Ricci-flat Kähler metric on $A$. On 
the other hand, it is well-known that any Ricci-flat metric on a torus is flat. Next, we show that 
$\Hol(\Xr,g)_x\cong \mathbb Z/2\mathbb Z$. Note that the natural surjection 
$\pi_1(X_{\reg})\twoheadrightarrow \Hol(\Xr,g)_x/\Hol^\circ(\Xr,g)_x$ together with the fact that
$(\gamma_{|\gamma^{-1}(X_{\reg})})^*g$ is flat yields a surjective map
$\mathbb Z/2\mathbb Z \twoheadrightarrow \Hol(\Xr,g)_x$. Suppose that $\Hol(\Xr,g)_x$ is trivial. Then, by Bochner principle, one would get a non-zero element $\sigma \in H^0(X,\Omega^{[1]}_X)$, and therefore a non-zero global $1$-form on $A$ invariant under the involution $a \mapsto -a$, yielding a contradiction. Thus
$$\Hol(\Xr,g)_x\cong \mathbb Z/2\mathbb Z.$$

As for the fundamental group, $X$ is simply-connected by Lefschetz theorem. Moreover, a degree two étale cover of $X_{\reg}$ is isomorphic to the complement of $16$ points in an abelian surface, therefore the fundamental group of this cover is $\mathbb Z^4$. 
\end{exmp}

\begin{exmp}[Symmetric square of a $K3$ surface]
This following example was already exhibited in \cite[Example 8.6]{GKP} as a variety with stable but not strongly stable tangent sheaf. 
Let $S$ be a $K3$ surface, and let $X:=S\times S/\langle i \rangle$ where $i$ is the natural involution
$(s_1,s_2)\mapsto (s_2,s_1)$
of $S\times S$. Recall from \cite[Section 6]{beauville83} that the Hilbert scheme $S^{[2]}$ parametrizing length $2$ zero-dimensional subschemes on $S$ is an irreducible symplectic manifold which admits a birational crepant morphism $S^{[2]}\to X$.
Then the following holds.

\medskip

\begin{itemize}
\item The variety $X$ admits a Kähler deformation to a smooth irreducible symplectic variety. 
\item The tangent sheaf $X$ is not strongly stable. 
\item Let $x \in X_{\reg}$. There exists a Ricci-flat Kähler metric $g$ on $X_{\reg}$ such that 
$$\Hol(X_{\reg},g)_x\cong \big(\SU(2)\times \SU(2)\big) \rtimes \mathbb Z/2\mathbb Z \quad \mbox{and} \quad \Hol^{\circ}(X_{\reg},g)_x\cong\SU(2)\times \SU(2).$$
\item We have $\pi_1(X)=\{1\}$ and $\pi_1(X_{\reg})\cong \mathbb Z/2\mathbb Z$. 
\end{itemize}

\medskip

The claim on the existence of a smoothing follows from the existence of the symplectic resolution 
$S^{[2]}\to X$ combined with a theorem of Namikawa \cite[Corollary 2]{Namikawa06} (see also Section~\ref{ex:symp}). 

The tangent sheaf $T_X$ is not strongly stable as $X$ admits a quasi-étale cover $\gamma\colon S \times S\to X$ of degree two and $\sT_{S\times S}$ is obviously not stable. 

Let us consider the two projections $p_i\colon S\times S\to S$ with $i=1,2$.
Given a Ricci-flat K\"ahler metric $g_S$ on $S$, $p_1^*g_S+p_2^*g_S$ defines a Kähler Ricci-flat metric on $S\times S$ that descends to a Ricci-flat Kähler metric $g$ on $X_{\reg}$. 
Let $y$ be a point on $S \times S$ such that $\gamma(y)=x$.
As the restricted holonomy is preserved by passing to an \'etale cover, we must have 
$$\Hol^{\circ}(X_{\reg},g)_x\cong \Hol^{\circ}(S\times S\smallsetminus \Delta,\gamma^*g)_{y}=\SU(2)\times \SU(2)$$
where $\Delta \subset S \times S$ denotes the diagonal.

As in the previous example, we have a surjective map 
$\mathbb Z/2\mathbb Z \twoheadrightarrow \Hol(\Xr,g)_x/\Hol^\circ(\Xr,g)_x$. Suppose that $\Hol(\Xr,g)_x^\circ=\Hol(\Xr,g)_x$. Then, by the Bochner principle, one would get two linearly independent reflexive $2$-forms on $X$, hence two independent $2$-forms on $S\times S$ invariant by the involution, yielding a contradiction. This shows that
$$\Hol(X_{\reg},g)_x\cong \big(\SU(2)\times \SU(2)\big) \rtimes \mathbb Z/2\mathbb Z.$$ 

By a theorem of Armstrong (see \cite{armstrong}), we have $\pi_1(X) =\{1\}$. As for $X_{\reg}$, it admits a degree two cover from $S\times S\smallsetminus \Delta$. Since $\codim \Delta =2$ and $S\times S$ is simply-connected, we conclude that $S\times S\smallsetminus \Delta$ is simply-connected as well, and hence 
$\pi_1(X_{\reg})\cong \mathbb Z/2\mathbb Z$.

\begin{rem}
As already mentioned,
Namikawa's theorem \cite[Corollary~2]{Namikawa06} provides smoothings by Kähler manifolds which are not projective. We do not know whether $X$ admits a projective smoothing by irreducible symplectic varieties.
\end{rem}
\end{exmp}

\providecommand{\bysame}{\leavevmode\hbox to3em{\hrulefill}\thinspace}
\providecommand{\MR}{\relax\ifhmode\unskip\space\fi MR }
\providecommand{\MRhref}[2]{%
  \href{http://www.ams.org/mathscinet-getitem?mr=#1}{#2}
}
\providecommand{\href}[2]{#2}

\end{document}